\documentclass[letter,fleqn,12pt]{article}
\usepackage[margin=1.87cm]{geometry}
\usepackage{bbold, epsfig}
\usepackage{latexsym, amsmath, euscript}
\usepackage{indentfirst}
\usepackage{enumerate}
\usepackage{tabularx}
\usepackage[utf8]{inputenc}
\usepackage{caption}
\usepackage{graphicx}
\usepackage{txfonts}
\usepackage{lipsum}
\usepackage{amssymb}

\usepackage{fancyhdr, lastpage}
\usepackage{zref-totpages}
\usepackage{hyperref}
\usepackage{xcolor}
\usepackage{cite}
\usepackage{mathtools}

\usepackage{authblk}
\newtheorem{theorem}{Theorem}[section]
\newtheorem{lemma}[theorem]{Lemma}
\newtheorem{proposition}[theorem]{Proposition}

\newtheorem{remark}[theorem]{Remark}

\newcommand{\R}{\mathbb R}
\newcommand{\N}{\mathbb N}
\newcommand{\ol}[1]{\overline{#1}}
\newcommand{\duality}[1]{\langle #1\rangle}

\newcommand{\bu}{{\textit{\textbf u}}}
\newcommand{\bv}{{\textit{\textbf v}}}
\newcommand{\bw}{{\textit{\textbf w}}}
\newcommand{\bg}{{\textit{\textbf g}}}
\newcommand{\ibf}{{\textit{\textbf f}}}

\newcommand{\bx}{{\textit{\textbf x}}}

\newcommand{\bn}{{\textit{\textbf n}}}
\newcommand{\bt}{{\textit{\textbf t}}}
\newcommand{\bH}{{\textbf{\textit H}}}
\newcommand{\bL}{{\textbf{\textit L}}}

\newcommand{\bD}{{\textbf{\textit D}}}

\newcommand{\cD}{{{\mathcal D}}}

\newcommand{\bcD}{{\boldsymbol{{\mathcal D}}}}

\newcommand{\pOmega}{\partial\Omega}
\newcommand{\bomega}{{\boldsymbol{{\omega}}}}
\newcommand{{{\bvarpi}}}{{\boldsymbol{{\varpi}}}}
\newcommand{\bphi}{{\boldsymbol{{\varphi}}}}
\newcommand{\bthetaup}{{\boldsymbol{{\varthetaup}}}}
\newcommand{\ds}[1]{{\displaystyle{#1}}}

{\newcounter{lequation}%
\newcounter{requation}%
\setcounter{lequation}{\vlue{equation}}%
\setcounter{requation}{\value{equation}}%
\addtocounter{requation}{3}}
{\setcounter{equation}{\value{requation}}}

\fontfamily{palatino}

\begin{document}
\fontsize{13}{16}\selectfont 

\begin{center}
{\bf\Large 
The structure of the solution to the generalized Stokes equations and
a new method for solving them\footnote{A version of the paper has been posted in arxiv.org}}
\end{center}

\begin{center}
{\Large Arian Novruzi}
\vspace*{2mm}

Department of Mathematics and Statistics, University of Ottawa, ON, Canada
\\
E-mail: \texttt{novruzi@uottawa.ca}
\end{center}


\begin{abstract}
We show that the velocity solution $\bu$ to the generalized Stokes equations with Dirichlet boundary conditions in dimension two and three is the superposition of two  divergence-free velocity fields  
$\bomega$ and $\bthetaup$, $\bu=\bomega+\bthetaup$. Both $\bomega$ and $\bthetaup$ 
solve screened Poisson vector equations with right hand side respectively 
the rotational component in the Helmholtz decomposition of the external force,  
and the gradient of a harmonic scalar function $q$ called ``solid pressure". 
Similarly, the fluid pressure $p$ is the superposition of the potential $\pi$ 
in the Helmholtz decomposition of the external force, and of $q$, $p=\pi+q$.
It turns out that regardless the pressure $q$,  $\bomega+\bthetaup$ is divergence-free and satisfies the generalized Stokes equations, except the normal component of the Dirichlet boundary condition.
The role of $q$, which solves a  boundary linear equation with a self-adjoint 
definite positive operator, is to constrain 
$\bomega+ \bthetaup$ to satisfy even the normal component of the Dirichlet boundary condition.

The method is attractive for solving numerically the time dependent Navier-Stokes equations
in two and three dimensions. It decouples the velocity from the pressure and solves the incompressibility constraint as a byproduct 
of unconstrained screened Poisson vector equations for $\bomega$ and $\bthetaup$, equipped
with appropriate boundary conditions.
In the two last sections we present two methods for solving numerically the generalized Stokes equations, and some results to demonstrate the efficiency of our approach.
\end{abstract}

\section{Introduction}
Let  $d=2$ or $3$, $\Omega\subset\R^d$ be a simply connected open bounded set of class $C^{1,1}$,
$T>0$, and consider the following Navier-Stokes equations
%
\begin{subequations}\label{e:NS}
\begin{align}
\rho\frac{D\bu}{Dt} - \mu \Delta \bu + \nabla p&={\ibf}\;\; in\;\;\Omega,
\\
\nabla\cdot\bu&=0\;\; in\;\; \Omega,
\\
\bu&=\bg\;\; on\;\; \partial\Omega,
\\
\bu(\cdot,0)&=\bu_0\;\; in\;\;\Omega,
\end{align}
\end{subequations}
where
\begin{subequations}
\begin{align}
\bu&=\bu(\bx,t)\in\R^d,
\;\;
p=p(\bx,t)\in\R,\quad \mbox{ are the unknowns,}
\\
{\ibf}&:\Omega\times[0,T]\mapsto\R^d\mbox{ and}
\\
{\bg}&:\pOmega\times[0,T]\mapsto\R^d\mbox{ are given time dependent vector functions, and}
\\
\frac{D\bu}{Dt}
&=\partial_t \bu + (\nabla\cdot \bu)\bu,
\quad
(\bx,t)\in \Omega\times(0,T),
\quad
\rho>0,\; \mu>0.
\end{align}
\end{subequations}
The common strategy for solving numerically \eqref{e:NS} is by first discretizing it in time and then solving the resulting PDE in $\Omega$.
Namely, given $N\in\N$, let $\Delta t=T/N$ and $t_n=n\Delta t$, $n=1,2,\ldots,N$.
Then we look for an approximation of $\bu$ and $p$ at times $t_n$. 
After discretizing the nonlinear term, for example with the first order backward approximation
or with the backward approximation along characteristics, 
see \cite{pironneau-1982-ns+char,russell-1982-ns+char},
one ends up with the following model problem 
\begin{subequations}\label{e:bu-all}
\begin{align}
\alpha\bu
- \mu \Delta \bu
+ \nabla p
&={\ibf}
\;\; in\;\;\Omega,		\label{e:bu-Om}
\\
\nabla\cdot\bu&=0\;\; in\;\; \Omega,	\label{e:div(bu)-Om}
\\
\bu&=\bg\;\; on\;\; \partial\Omega,	\label{e:bu-pOm}
\end{align}
\end{subequations}
where $\bu=\bu(\bx)$, $p=p(\bx)$ are the unknown functions, $\bg=\bg(\bx)$ is a given 
function and $\mu>0$, $\alpha\geq0$. We will refer to \eqref{e:bu-all} as ``generalized Stokes equations/problem".
This is the problem that we will study in this paper. 

The main difficulty for solving \eqref{e:bu-all} comes from the coupling of \eqref{e:bu-Om}
with \eqref{e:div(bu)-Om}, which typically is addressed
by decoupling the velocity from the pressure - an idea that dates back to 1960s to 1970s.
Without discussing the immense literature on Stokes equations, we single out Uzawa and 
projection methods, \cite{chorin-1968,temam-1969,uzawa-1958,glowinski-1984,bristeau+al-1987}.
Uzawa method solves the constraint \eqref{e:div(bu)-Om} iteratively, and between iterations solves the velocity components separately, with an overall slow convergence rate which can be accelerated with appropriate preconditioners, \cite{bristeau+al-1987}[Section 4.3.4].
The projection method addresses the constraint \eqref{e:div(bu)-Om} by first computing an intermediate velocity counting for the external forces, followed by a projection step 
where the divergence-free velocity is computed along with the fluid pressure.
In this method a boundary layer for the pressure appears because the method leads to an
artificial zero Neumann boundary condition for the pressure, see \cite{guermond+al-2006}.

Note that in these methods the pressure solves an equation in the whole domain.
In \cite{glowinski+pironneau-1978,glowinski+pironneau-1979}  the velocity and the pressure are decoupled as well, and the solution to Stokes equations is reduced to the solution of a well posed boundary equation, which equivalently is a constraint for the divergence-free condition.

Our method is different from these methods in the way it solves the divergence-free constraint \eqref{e:div(bu)-Om}.
We construct the solution $\bu$ of \eqref{e:bu-all} as the sum of
two divergence-free velocities,  $\bu = \bomega + \bthetaup$.
The velocity $\bomega$ solves $\alpha\bomega - \mu\Delta\bomega=\bvarpi$, and
$\bthetaup=\bthetaup(q)$ solves $\alpha\bthetaup - \mu\Delta\bthetaup =-\nabla q$,
where $\bvarpi$ is the rotational component of the Helmholtz decomposition of 
the external force,\,  $\ibf=\bvarpi+\nabla\pi$, and 
$q$ is a harmonic function.
It turns out that regardless $q$,
$\nabla\cdot\bomega$ and $\nabla\cdot\bthetaup$ satisfy the equation 
$\alpha z - \mu\Delta z=0$. 
We impose appropriate boundary conditions for $\bomega$ and $\bthetaup$, easy to implement in
their weak form equations,  
which imply that both $\bomega$ and $\bthetaup$ are divergence-free regardless $q$, and 
$\bu=\bomega+\bthetaup$ solves \eqref{e:bu-all}, except the normal component of \eqref{e:bu-pOm}.
Then, we choose the boundary values of $q$ such that even the normal component of 
\eqref{e:bu-pOm} is satisfied. It turn out that this is equivalent to
$Aq:=-\bthetaup(q)\cdot\bn=(\bomega-\bg)\cdot\bn$, with $A$ a boundary self-adjoint definite-positive operator,  see Section \ref{sec:not+main} for more details.
We also note  that our method does not rely on the 
well-known inf-sup condition, because both $\bomega$  and $\bthetaup(q)$ solve well-posed 
unconstrained screened Poisson equation $\alpha I - \mu\Delta$.

It is interesting to note that the harmonic function $q$ can be considered as 
the model of the pressure force exerted by the solid boundary to the fluid.
Its role is  to constrain the velocity $\bthetaup(q)$ so that the sum 
$\bomega+\bthetaup(q)$ satisfies the  normal component of the Dirichlet boundary condition.
It tuns out that up to a constant, $p=\pi+q$.

Our method applies to time-dependent Navier-Stokes equations.
In a forthcoming paper we will address in details the numerical analysis of the method
and its application to time-dependent Navier-Stokes equations.

The structure of the paper is as follows.
In Section \ref{sec:not+main} we introduce the notations, working spaces and present the main results.
In Section \ref{sec:prelim} we present some known, and prove some preliminary, results.
In Section \ref{sec:proof} we prove the main results, 
in Section \ref{sec:alg} we present a new algorithm for solving \eqref{e:bu-all},
in Section \ref{sec:num-results}
we present some numerical results and in Section \ref{sec:concl} we draw some final  conclusions.

\section{Notations and the main result}\label{sec:not+main}
All along this paper boldface letters and symbols will denote vectors, vector functions and vector spaces.
Furthermore, $\Omega\subset\R^d$ denotes an open bounded simply connected set of class $C^{1,1}$,
and $\bn$ denote its unitary normal exterior vector on $\pOmega$.
We will consider a $C^{0,1}(\overline{\Omega})$ extension of $\bn$, which 
we will still denote by $\bn$. For example, we can take 
\begin{equation}\label{e:n}
\bn(\bx)=\frac{\nabla \delta(\bx)}{|\nabla \delta(\bx)|}(1-\varphi(\bx)),\quad
\mbox{where }\;\;
\delta(\bx)=-{\rm dist}(\bx,\pOmega),\;\;\; \varphi\in\cD(\Omega).
\end{equation}
%
The spaces
\begin{subequations}\label{e:spaces(Omega)}
\begin{align}
&\cD(\Omega),\; \cD(\overline{\Omega}),
&&\;\; \bcD(\Omega)=(\cD(\Omega))^d,\; \bcD(\overline{\Omega})=(\cD(\overline{\Omega}))^d,
\\
&L^2(\Omega),\; H^k(\Omega),\; H^k_0(\Omega), 
&&\;\;\bL^2(\Omega)=(L^2(\Omega))^d,\; 
\bH^k(\Omega)=(H^k(\Omega))^d,\; \bH^k_0(\Omega)=(H^k_0(\Omega))^d,	\label{e:L2,Hk}
\end{align}
\end{subequations}
with $0<k\in\mathbb N$ are classical. Each space in \eqref{e:L2,Hk} equipped with the obvious inner product is an Hilbert space.
The dual of $H^k_0(\Omega)$, resp. $\bH^k_0(\Omega)$, will be denoted by
$H^{-k}(\Omega)$, resp. $\bH^{-k}(\Omega)$.
%
%

For vector functions 
$\bu=(u_1,\ldots,u_d)$, $\bv=(v_1,\ldots,v_d)$ in $\bcD(\ol{\Omega})$ we will write
\begin{subequations}\label{e:Dnu,Dt*u}
\begin{align}
\hspace*{-8mm}
\nabla\bu &= [\partial_j u_i]\in\R^{d\times d},
\quad
&
\nabla\cdot\bu 
&= 
\sum_{i=1}^d\partial_i u_i
\in\R,
\quad
\nabla\times\bu
=
\nabla\bu - {^t}\nabla\bu
\in\R^{d\times d}
\quad \mbox{in}\;\, \Omega,
\\
\hspace*{-8mm}
\bu\cdot\bv
&=
\sum_{i=1}^d u_iv_i,\quad
&
\nabla\bu:\nabla\bv
&=
{\rm tr}(\nabla\bu\cdot{^t\nabla}\bv)
=
{\rm tr}({^t\nabla}\bu\cdot{\nabla}\bv)\quad \mbox{in}\;\, \Omega,
\\
\bu_\bn &= (\bu\cdot\bn)\bn,\quad
&
\bu_\bt 
&= \bu - \bu_\bn\quad\mbox{on}\;\, \pOmega,
\\
\hspace*{-8mm}
\partial_\bn\bu &= \nabla\bu\cdot\bn,\quad
&\nabla_\bt\cdot\bu &= \nabla\cdot\bu - (\partial_\bn\bu\cdot\bn)\quad\mbox{on}\;\, \pOmega,
\end{align}
\end{subequations}
where ${\rm tr}(\cdot)$ is the trace of the matrix $(\cdot)$
and ${^t}\nabla\bu$ is the transposed of $\nabla\bu$.
We note that $\nabla_\bt\cdot\bu$, called ``tangential divergence of $\bu$", is well-defined 
regardless\footnote{Indeed, if $\bu=(u_1,\ldots,u_d)=0$ on $\partial\Omega$ then 
$\nabla u_i=z_i\bn$ on $\partial\Omega$, for a certain function $z_i$, $i=1,\ldots,d$. 
Then
$\nabla_\bt\cdot\bu=\sum_{i=1}^d z_in_i-\sum_{i,j=1}^d z_in_jn_jn_i=0$} the $C^{0,1}$ extension of $\bn$ and $\bu$.

%

The space $\bH^1_\bn(\Omega)$, which plays a central role in our analysis, is defined by
\begin{equation}\label{e:H1n(Om)}
\bH^1_\bn(\Omega)
=\{\varphi\bn + \bv,\; \varphi\in H^1({\Omega}),\; \bv\in\bH^1_0(\Omega)\}.
\end{equation}
Clearly $\bH^1_\bn(\Omega)\subset \bH^1(\Omega)$, 
and equipped with the $\bH^1(\Omega)$ inner product it is an Hilbert space.
The dual space of $\bH^1_\bn(\Omega)$ will be denoted by $\bH^{-1}_\bn(\Omega)$.
Note that $\bH^1_\bn(\Omega)$ can  equivalently be defined as 
\begin{subequations}\label{e:H1n(Om)~}
\begin{align}
\bH^1_\bn(\Omega)
&={\rm cl}_{\bH^1(\Omega)}
\{\varphi\bn + \bv,\; \varphi\in \cD(\ol{\Omega}),\; \bv\in\bcD(\Omega)\},
\quad\mbox{or}
\\
\bH^1_\bn(\Omega)
&=\{\bu\in\bH^1(\Omega),\;\; \bu_\bt=0\;\,\mbox{on}\; \pOmega\}.
\end{align}
\end{subequations}

We will also consider the boundary Hilbert space
$H^{1/2}(\partial\Omega)$, its dual $H^{-1/2}(\partial\Omega)$,
and the following quotient Hilbert spaces
\begin{equation}\label{e:spaces(pOmega)}
H^{1/2}_0(\partial\Omega)=
\left\{p\in H^{1/2}(\partial\Omega), \int_{\partial\Omega}pd\sigma=0\right\}, 
\quad
H^{-1/2}_0(\partial\Omega)=
(H^{1/2}_0(\partial\Omega))'.
\end{equation}

Now let us explain our method and give the main results.
Given $\ibf\in\bL^2(\Omega)$ with $\nabla\cdot\ibf\in L^2(\Omega)$,  let $\bvarpi$ and $\pi$~be~defined~by
\begin{subequations}\label{e:f=Dp+w}
\begin{align}
{\ibf}=\bvarpi + \nabla \pi ,\quad\mbox{with}\quad
&\pi\in H^1(\Omega),\quad
\Delta \pi =\nabla\cdot{\ibf}\;\;
\mbox{in}\;\;\Omega,\;\;
\partial_\bn \pi={\ibf}\cdot\bn\;\;\mbox{on}\;\;\pOmega,
\\
&\bvarpi\in\bL^2(\Omega),\;\;\
\nabla\cdot\bvarpi=0\;\;
\mbox{in}\;\;\Omega,\;\; 
\bvarpi\cdot{\bf n}=0\;\;\mbox{on}\;\;\pOmega.
\end{align}
\end{subequations}
We note that the pair $(\bvarpi,\pi)$ satisfying \eqref{e:f=Dp+w} exists and is unique 
(with pressure up to a constant) 
because  $\ibf\cdot\bn\in H^{-1/2}(\pOmega)$, see \cite[Theorem 2.5]{girault+raviart-1986}.
Then we consider ${\bomega}\in \bH^1(\Omega)$ satisfying
\begin{subequations}\label{e:uf-all}
\begin{align}
\alpha {\bomega}
- \mu \Delta {\bomega}
&=
\bvarpi
\;\; in\;\;\bL^2(\Omega),		\label{e:uf-Om}
\\
{\bomega}_\bt
&=\bg_\bt\;\; in\;\; \bH^{1/2}(\partial\Omega), \label{e:uft-pOm}
\\
\nabla\cdot{\bomega}
&=0\;\; in\;\; \bH^{-1/2}(\partial\Omega).\label{e:dnufn-pOm}
\end{align}
\end{subequations}
For $q\in H^{-1/2}_0(\pOmega)$ we consider its harmonic extension in $\Omega$, still denoted 
by $q$,
\begin{equation}\label{e:q}
q\in L^2(\Omega),\;\;
\Delta q=0\;\;\mbox{in}\;\; \cD'(\Omega).
\end{equation}
Then $\bthetaup=\bthetaup(q) \in\bH^1_\bn(\Omega)$ is defined by
\begin{subequations}\label{e:uq-all}
\begin{align}
\alpha \bthetaup
- \mu \Delta \bthetaup
&=
-\nabla q
\;\; in\;\;\bH^{-1}_\bn(\Omega),		\label{e:uq-Om}
\\
\bthetaup_\bt
&=0\;\; in\;\; \bH^{1/2}(\partial\Omega),	\label{e:uqt-pOm}
\\
\nabla\cdot\bthetaup
&=0\;\; in\;\; \bH^{-1/2}(\partial\Omega).	\label{e:dnuqn-pOm}
\end{align}
\end{subequations}
We note that formally, both $\nabla\cdot\bthetaup$ and 
$\nabla\cdot\bomega$ satisfy the equation
$\alpha z - \mu\Delta z=0$ in $L^2(\Omega)$, because $\nabla\cdot\bvarpi=\Delta q=0$.
Furthermore, from the boundary conditions
\eqref{e:dnufn-pOm} and \eqref{e:dnuqn-pOm} we get
$\nabla\cdot\bomega=\nabla\cdot\bthetaup=0$ on $\pOmega$ (see next section)
which implies $z=0$ in $\Omega$, so $\nabla\cdot\bthetaup=\nabla\cdot\bomega=0$
and $\nabla\cdot(\bomega + \bthetaup)=0$ in $\Omega$.
Furthermore, summing  \eqref{e:uf-Om} with \eqref{e:uq-Om}
and  \eqref{e:uft-pOm}  with \eqref{e:uqt-pOm} gives
\begin{subequations}\label{e:uq+uf-all}
\begin{align}
\alpha (\bomega+\bthetaup)
- \mu \Delta (\bomega+\bthetaup)
+\nabla (\pi+q)
&=\ibf\;\; in\;\;\Omega,		\label{e:uq+uf-Om}
\\
\nabla\cdot(\bomega+\bthetaup)
&=0\;\; in\;\; \Omega,
\\
(\bomega+\bthetaup)_\bt
&=\bg_\bt\;\; on\;\; \partial\Omega.	\label{e:uqt+uft-pOm}
\end{align}
\end{subequations}
So, $\bu=\bomega+\bthetaup(q)$ and $p=\pi+q$ satisfy \eqref{e:bu-Om}, \eqref{e:div(bu)-Om} and the tangential component of \eqref{e:bu-pOm}, regardless the choice of $q$. 
We choose $q\in H^{-1/2}_0(\pOmega)$ such that
$\bu=\bomega+\bthetaup(q)$ satisfies also the normal component of \eqref{e:bu-pOm}, i.e.
$(\bomega+\bthetaup(q))\cdot\bn=\bg\cdot\bn$, which would imply that 
$\bu=\bomega+\bthetaup(q)$~solves~\eqref{e:bu-all}.

\noindent
Our main results are given by the following theorem.
\begin{theorem}\label{th:main}
Let $\Omega\subset\R^d$ with $d=2,3$, be a $C^{1,1}$ simply connected open bounded set,
$\ibf\in\bL^2(\Omega)$ with $\nabla\cdot\ibf\in L^2(\Omega)$,
and $\bg\in \bH^{1/2}(\pOmega)$ with $\int_{\pOmega} \bg\cdot\bn d\sigma=0$. 
Then the followings hold.
\begin{enumerate}
\item[(a)] 
The problem \eqref{e:uf-all} has a unique solution $\bomega\in \bH^1(\Omega)$ with 
$\nabla\cdot\bomega=0$ in $L^2(\Omega)$.
\item[(b)]
For every $q\in H^{-1/2}_0(\partial\Omega)$, let us denote still by
$q$ its harmonic $D(\Delta,L^2(\Omega))$ extension.
Then the problem \eqref{e:uq-all} has a unique solution
$\bthetaup=\bthetaup(q)\in\bH^1_\bn(\Omega)$ with $\nabla\cdot\bthetaup=0$  in $L^2(\Omega)$.
\item[(c)]
The map 
$q\in  H^{-1/2}_0(\partial\Omega)
\mapsto 
Aq:=-\bthetaup(q)\cdot{\bf n}\in  H^{1/2}_0(\partial\Omega)$,
is a self-adjoint, positive-definite isomorphism as follows
\begin{subequations}
\begin{align}
\duality{p,Aq}&=\duality{q,Ap},\quad\forall p,q\in H^{-1/2}_0(\pOmega),	\label{e:qAq=qAp}
\\
m\|q\|^2_{H^{-1/2}_0(\pOmega)}
&\leq
\duality{q,Aq}
\leq
M\|q\|^2_{H^{-1/2}_0(\pOmega)},\quad \forall q\in H^{-1/2}_0(\pOmega),	\label{e:m<(Aq,q)<M}
\end{align}
\end{subequations}
with $m=m(\Omega,\alpha,\mu)>0$, $M=M(\Omega,\alpha,\mu)>0$.
\item[(d)]
There exists a unique $q\in  H^{-1/2}_0(\partial\Omega)$ solving 
$Aq=(\bomega-\bg)\cdot\bn$, and 
the unique solution of \eqref{e:bu-all} is given by $\bu=\bomega+\bthetaup(q)$.
\end{enumerate}
\end{theorem}

\begin{remark}\label{r:main+comments}
From (a) and (b) of Theorem \ref{th:main} we get two ways of constructing divergence-free functions with a prescribed boundary tangential component, each by solving a vector 
screened Poisson equation with a divergence-free or gradient of a harmonic function right hand side term.
The components of these vector screened Poisson equations are coupled only on the boundary, 
which is attractive from the numerical viewpoint.

\noindent
From (c), as the operator $A$ is self-adjoint and positive definite, we can use 
direct or iterative methods to solve numerically 
the equation $Aq=(\bomega-\bg)\cdot\bn$.
In Section \ref{sec:alg} we will present two methods for solving the equation
$Aq=(\bomega-\bg)\cdot\bn$, so for solving \eqref{e:bu-all}, and in Section 
\ref{sec:num-results} we will present some numerical results.
\end{remark}

\section{Preliminary results}\label{sec:prelim}

\noindent
In this section we will present some well-known results and prove some preliminary results 
that we will use in the next sections.
The following is a classical regularity result, see \cite[Corollary 2.2.2.4, Corollary 2.2.2.6]{grisvard-1985}.
\begin{lemma}\label{l:D,N-iso}
The maps 
$(\Delta,\gamma_0):p\in H^2(\Omega)
\mapsto 
\{\Delta p,\gamma_0 p\}\in L^2(\Omega)\times H^{3/2}(\partial\Omega)$ and
$(\Delta,\gamma_1):p\in H^2(\Omega)
\mapsto 
\{\Delta p,\gamma_1 p\}\in L^2(\Omega)\times H^{1/2}(\partial\Omega)$
are isomorphisms,
where $\gamma_0p$ is the boundary trace of $p$ on $\pOmega$,
and $\gamma_1p$ is the boundary trace of $\partial_\bn p:=\nabla p\cdot\bn$ on $\pOmega$.
\end{lemma}
Note that in the following, whenever there is no ambiguity, we will write $p$, 
resp, $\partial_\bn p$, instead of $\gamma_0p$, resp. $\gamma_1p$.
%
%
%
%
We will need the following spaces equipped with their inner-products
\begin{subequations}
\begin{align}
\hspace*{-7mm}
D(\Delta,H^1(\Omega))
&=\left\{p\in H^1(\Omega),\; \Delta p\in L^2(\Omega)\right\},
\quad
(p,q)_{D(\Delta,H^1(\Omega))}
=
(p,q)_{H^1(\Omega)}+ (\Delta p,\Delta q)_{L^2(\Omega)},	\label{e:D(D,H1)}
\\
\hspace*{-7mm}
D(\Delta,L^2(\Omega))
&=
\left\{p\in L^2(\Omega),\; \Delta p\in L^2(\Omega)\right\},
\quad
(p,q)_{D(\Delta,L^2(\Omega))}
=
(p,q)_{L^2(\Omega)}+ (\Delta p,\Delta q)_{L^2(\Omega)}.	\label{e:D(D,L2)}
\end{align}
\end{subequations}

\noindent
In order to simplify the notations, in the following we will drop 
the space underscripts in the duality brackets $\duality{\cdot,\cdot}$, whenever the spaces
are always clear from the context.
 \begin{lemma}\label{l:D(A;L2,H1)}
The following results hold.
\begin{itemize}
\item[(a.1)]
The space $D(\Delta,H^1(\Omega))$ equipped with the  inner-product \eqref{e:D(D,H1)} 
is an Hilbert space and 
$\cD(\overline{\Omega})$ is dense in $D(\Delta,H^1(\Omega))$.
\item[(a.2)]
The map 
$(\Delta,\partial_\bn):p\in \cD(\overline{\Omega})\mapsto (\Delta p,\partial_\bn p)\in \cD(\overline{\Omega})\times \cD(\partial\Omega)$,
with $\partial_\bn p=\nabla p\cdot\bn$,
extends continuously to an isomorphism $(\Delta,\partial_\bn):D(\Delta,H^1(\Omega))\mapsto 
L^2(\Omega)\times H^{-1/2}(\partial\Omega)$  and the following formula of integration by parts  holds
\begin{equation}\label{e:duality(D(L;H1)}
\duality{\partial_\bn p, q}
=
\int_\Omega (q\Delta p + (\nabla p\cdot\nabla q)) {d\bx},
\quad
\forall p\in D(\Delta, H^1(\Omega)),\;\; 
\forall q\in H^1(\Omega).
\end{equation}
\item[(b.1)]
The space $D(\Delta,L^2(\Omega))$  equipped with the inner-product \eqref{e:D(D,L2)}
is an Hilbert space and $\cD(\ol{\Omega})$ is dense in $D(\Delta,L^2(\Omega))$.
\item[(b.2)]
The map 
$(p,\partial_\bn):p\in \cD(\overline{\Omega})\mapsto (p,\partial_\bn p)
\in  \cD(\partial\Omega)\times C^{0,1}(\pOmega)$
extends continuously from $D(\Delta,L^2(\Omega))$ to 
$H^{-1/2}(\pOmega)\times H^{-3/2}(\partial\Omega)$  and the following formula  of 
integration by parts holds
\begin{equation}\label{e:duality(D(L;L2)}
\duality{p,\partial_\bn q} - \duality{\partial_\bn p,q}
=
\int_\Omega (p\Delta q - q\Delta p) {d\bx},\quad
\forall p\in D(\Delta,L^2(\Omega)),\;\; 
\forall q\in H^2(\Omega).
\end{equation}
Furthermore, there exists $C=C(\Omega)>0$ such that
\begin{equation}\label{e:|q|H-1/2<C|q|L^2(Delta)}
\|p\|^2_{H^{-1/2}(\partial\Omega)}
\leq 
C(\Omega)\|p\|^2_{D(\Delta,L^2(\Omega))},\quad \forall p\in D(\Delta,L^2(\Omega)).
\end{equation}
\end{itemize}
\end{lemma}
{\bf Proof}.
For the proof see \cite{lions+magenes-1961}.
Alternatively, for (a) see \cite[p. 59, 62]{grisvard-1985}.
For (b.1) and first part of (b.2)  
see \cite[p. 54, 57]{grisvard-1985}.
For \eqref{e:|q|H-1/2<C|q|L^2(Delta)} see \cite[Theorem 2.5.2.1]{grisvard-1985}. 

We will provide a direct proof of \eqref{e:|q|H-1/2<C|q|L^2(Delta)} 
because we will use it again in the following. 
Let $p\in\cD(\ol{\Omega})$,  $g\in\cD(\pOmega)$ and $\delta(x)=-{\rm dist}(x,\pOmega)$. 
If $\varphi=g\delta$ then $\varphi\in C^{1,1}(\ol{\Omega})$ and
$\varphi=0$, $\partial_\bn\varphi=g$ on $\pOmega$, which follow easily from 
$\delta\in C^{1,1}(\ol{\Omega})$ and $\nabla \delta=\bn$ on $\pOmega$.
Furthermore, we consider a $H^1(\Omega)$ extension of $g$, denoted by the same letter, 
and satisfying $g-\Delta g=0$. 
It is well known that $g\in H^2(\Omega)$ and 
$\|g\|_{H^{1/2}(\pOmega)}=\|g\|_{H^1(\Omega)}$, see \cite{girault+raviart-1986}.
Then by applying \eqref{e:duality(D(L;L2)} with $q=\varphi$ we get
\begin{subequations}
\begin{align*}
\duality{p,g}
=
\duality{p,\partial_\bn \varphi}
&=
\int_\Omega (p\Delta\varphi - \varphi\Delta p)d\bx \\
&=
\int_\Omega (p(g(\delta + \Delta \delta) + 2\nabla g\cdot\nabla \delta) - g\delta\Delta p)d\bx ;\quad\mbox{hence}\\
|\duality{p,g}|
&\leq 
C\|p\|_{D(\Delta,L^2(\Omega))}\|g\|_{H^1(\Omega)}\\
&\leq
C\|p\|_{D(\Delta,L^2(\Omega))}\|g\|_{H^{1/2}(\pOmega)},
\quad\forall p, g\in\cD(\partial\Omega),
\end{align*}
\end{subequations}
with $C$ independent of $p$.
Then \eqref{e:|q|H-1/2<C|q|L^2(Delta)} follows from the last inequality and the density of
$\cD(\ol{\Omega})$ in $D(\Delta,L^2(\Omega))$ and $H^1(\Omega)$.
\hfill$\Box$

%

\begin{lemma}\label{l:div(u),H1L2}
Let us define 
\begin{equation}\label{e:bD^1_Delta}
\hspace*{-5mm}
\bD(\Delta,\bH^1(\Omega))
=\{\bu\in\bH^1(\Omega),\; \Delta\bu\in \bL^2(\Omega)\},
\quad
(\bu,\bv)_{\bD(\Delta,\bH^1(\Omega))}
=
(\bu,\bv)_{\bH^1(\Omega)}+ (\Delta\bu,\Delta\bv)_{\bL^2(\Omega)}.
\end{equation}
\begin{enumerate}
\item[(a)]
The space  $\bD(\Delta,\bH^1(\Omega))$  equipped with the inner-product in \eqref{e:bD^1_Delta} is an Hilbert space and $\bcD(\ol{\Omega})$ is dense in  $\bD(\Delta,\bH^1(\Omega))$.
\item[(b)]
The map
$\bu\in\bcD(\ol{\Omega})\mapsto\nabla\cdot\bu\in\cD(\partial\Omega)$ extends continuously 
from $\bD(\Delta,\bH^1(\Omega))$ to $H^{-1/2}(\partial\Omega)$.
Furthermore, there exists $C(\Omega)>0$ such that
\begin{equation}\label{e:|divu|H-1/2}
\|\nabla\cdot\bu\|_{H^{-1/2}(\partial\Omega)}
\leq 
C(\Omega)\|\bu\|_{\bD(\Delta,\bH^1(\Omega))},\quad \forall\bu\in\bD(\Delta,\bH^1(\Omega)).
\end{equation}
\item[(c)]
Similarly, the maps
$\bu\in\bcD(\ol{\Omega})\mapsto \partial_\bn\bu\cdot\bn\in C^{0,1}({\partial\Omega})$,
$\bu\in\bcD(\ol{\Omega})\mapsto \nabla_\bt\cdot \bu\in C^{0,1}({\partial\Omega})$ and
$\bu\in\bcD(\ol{\Omega})\mapsto \nabla_\bt\cdot\bu_\bt \in C^{0,1}({\partial\Omega})$,
extend continuously from 
$\bu\in\bD(\Delta,\bH^1(\Omega))$ to $H^{-1/2}(\partial\Omega)$ and
the following equalities hold for every $\bu\in\bD(\Delta,\bH^1(\Omega))$
\begin{subequations}\label{e:div(u)}
\begin{align}
\nabla\cdot\bu
&=\partial_\bn\bu\cdot\bn + \nabla_\bt\cdot \bu \quad  in\;\; H^{-1/2}(\partial\Omega),
\label{e:div(u)=Du*n+Dt*u}
\\
\nabla_\bt\cdot\bu 
&= 
\nabla_\bt\cdot\bu_\bt
+
(\nabla_\bt\cdot\bn)(\bu\cdot\bn)\quad in\;\; H^{-1/2}(\partial\Omega),\quad \mbox{and}
\label{e:Dt*u=Dt*ut+H(u*n)}
\\
\nabla_\bt\cdot\bn
&=
-\Delta \delta + \bn\cdot D^2\delta\cdot\bn\quad  in\;\; L^\infty(\pOmega).		\label{e:Dt*n=H}
\end{align}
Here $\nabla_\bt\cdot\bn$ is the mean curvature of $\pOmega$ and it is non negative for $\Omega$ convex.
\end{subequations}
\item[(d)]
Furthermore, for $\bu\in\bH^1_\bn(\Omega)$ with $\Delta\bu\in\bL^2(\Omega)$ we have
\begin{equation}\label{e:D*u=H(u*n)}
\nabla_\bt\cdot\bu_\bt=0\;\; and\;\;
\nabla\cdot\bu 
= 
\partial_\bn\bu\cdot\bn
+
(\nabla_\bt\cdot\bn)(\bu\cdot\bn) 
\quad in\;\; H^{-1/2}(\partial\Omega).
\end{equation}
\item[(e)]
Lastly, we have
\begin{equation}\label{e:<Dnu,phi>=<Dnu*n,phi*n>}
\duality{\partial_\bn\bu,\bphi}
=
\duality{\partial_\bn\bu\cdot\bn,\bphi\cdot\bn},
\quad
\forall \bu\in\bD(\Delta,\bH^1(\Omega)),\;\; \forall\bphi\in\bH^1_\bn(\Omega).
\end{equation}
\end{enumerate}
\end{lemma}
{\bf Proof}.
The proof of (a) is as in (a.1), Lemma \ref{l:D(A;L2,H1)}.
For (b), let $\bu\in\bcD(\ol{\Omega})$, $g\in\cD(\pOmega)$. 
As in (b.2), Lemma \ref{l:D(A;L2,H1)}, we consider the 
$H^1(\Omega)$ extension $g$, the distance function $\delta$ and set $\varphi=g\delta$. 
Applying \eqref{e:duality(D(L;L2)} with $p=\nabla\cdot\bu$ and $q=\varphi$, and after integrating by parts we get
\begin{align*}
\duality{\nabla\cdot\bu,g}
=
\duality{\nabla\cdot\bu,\partial_\bn \varphi}
&=
\int_\Omega 
((\nabla\cdot\bu)\Delta \varphi - \varphi\Delta(\nabla\cdot\bu))\, {d\bx}
\\
&=
\int_\Omega 
((\nabla\cdot\bu)(g(\delta+\Delta \delta) + 2(\nabla \delta\cdot\nabla g)) + 
(\nabla (g\delta)\cdot \Delta\bu)).
\end{align*}
This implies
\begin{align}
|\duality{\nabla\cdot\bu,g}|
\leq
C(\Omega)\|\bu\|_{\bD(\Delta,\bH^1(\Omega))}\|g\|_{H^1(\Omega)} 
\leq
C(\Omega)\|\bu\|_{\bD(\Delta,\bH^1(\Omega))}\|g\|_{H^{1/2}(\pOmega)},	\label{e:|<D*u,g>|}
\end{align}
which combined with the density of $\bcD(\ol{\Omega})$ in $\bD(\Delta,\bH^1(\Omega))$
and of $\cD(\ol{\Omega})$ in $H^1(\Omega)$ proves~(b)~and~\eqref{e:|divu|H-1/2}.

For the first part of (c), we apply (a.2), Lemma  \ref{l:D(A;L2,H1)} as follows. The map 
$u_i\in\cD(\ol{\Omega})\mapsto \partial_\bn u_i\in C^{0,1}(\partial\Omega)$,
$i=1,\ldots,d$, extends continuously from 
$D(\Delta,H^1(\Omega))$ to $H^{-1/2}(\partial\Omega)$ and
\begin{equation}\label{e:dnu.n}
\duality{\partial_\bn u_i,\varphi}
=
\int_\Omega (\varphi\Delta u_i + \nabla u_i\cdot\nabla \varphi)\,d\bx,
\quad\forall u_i\in D(\Delta,H^1(\Omega)),\;\; \varphi\in H^1(\Omega),\;\; i=1,\ldots,d.
\end{equation}
Now let $\bu=(u_1,\ldots,u_d)\in\bcD(\ol{\Omega})$.
For $g\in H^{1/2}(\pOmega)$ we consider its extension in $\Omega$ satisfying 
$g-\Delta g=0$ so that 
$\|g\|_{H^{1/2}(\pOmega)}=\|g\|_{H^1(\Omega)}$. 
Then \eqref{e:dnu.n} implies
\begin{align}
\duality{\partial_\bn\bu\cdot\bn,g}
=
\sum_{i=1}^d
\duality{\partial_\bn u_i,n_i g}
&=
\sum_{i=1}^d
\int_\Omega (n_i g\Delta u_i + \nabla u_i\cdot\nabla (n_i g))\,d\bx
\nonumber \\
&=
\int_\Omega 
(
g(\bn\cdot\Delta\bu) + \nabla\bu:\nabla\bn)) + 
\bn\cdot\nabla\bu\cdot\nabla g)\,d\bx;
\quad
\mbox{so }\\
|\duality{\partial_\bn\bu\cdot\bn, g}|
&\leq
C\|\bu\|_{\bD(\Delta, \bH^1(\Omega))}\|g\|_{H^1(\Omega)}
\nonumber\\
&\leq
C\|\bu\|_{\bD(\Delta, \bH^1(\Omega))}\|g\|_{H^{1/2}(\pOmega)}.	\label{e:|dnu.n|<,1}
\end{align}
By density, the estimate \eqref{e:|dnu.n|<,1} holds for 
$\bu\in\bD(\Delta,\bH^1(\Omega))$,
and proves (c) for the extension of the map $\bu\mapsto\partial_\bn\bu\cdot\bn$.
Furthermore, for $\bu\in\bcD(\ol{\Omega})$ we have
$\nabla\cdot\bu = \partial_\bn\bu\cdot\bn + \nabla_\bt\cdot\bu$, which combined
with (a), (b), \eqref{e:|<D*u,g>|} and  \eqref{e:|dnu.n|<,1} proves the claim (c) for the map 
$\bu\mapsto\nabla_\bt\cdot\bu$ and \eqref{e:div(u)=Du*n+Dt*u}.

For the remaining part of (c) we note that \eqref{e:Dt*u=Dt*ut+H(u*n)} holds for 
$\bu\in\bcD(\ol{\Omega})$ because we have
\begin{align*}
\nabla_\bt\cdot\bu
&= 
\nabla_\bt\cdot(\bu_\bt + (\bu\cdot\bn)\bn) 	
=
\nabla_\bt\cdot\bu_\bt + 
\nabla\cdot((\bu\cdot\bn)\bn) - \partial_\bn((\bu\cdot\bn)\bn)\cdot\bn 	\nonumber\\
&=
\nabla_\bt\cdot\bu_\bt +
(\nabla_\bt\cdot\bn)(\bu\cdot\bn).
\end{align*}
Then \eqref{e:Dt*u=Dt*ut+H(u*n)} follows from the density of $\bcD(\ol{\Omega})$ in 
$\bD(\Delta,\bH^1(\Omega))$.
The formula \eqref{e:Dt*n=H} is classical, see for example \cite{gilbarg+trudinger-2001}, 
and it follows as below
\begin{align*}
\nabla_\bt\cdot\bn
&=
\sum_{i=1}^d\partial_i n_i - \partial_\bn\bn\cdot\bn
=
\sum_{i=1}^d\partial_i\left(-\frac{\partial_i \delta}{|\nabla \delta|}\right) 
- 
\frac{1}{2}\partial_\bn(|\bn|^2)
\\
&=
-
\sum_{i=1}^d
\frac{\partial_{ii} \delta}{|\nabla \delta|}
+
\sum_{i,j=1}^d
\partial_i\delta\partial_j\delta\frac{\partial_{ij}\delta}{|\nabla \delta|^3}
=
-\Delta \delta + \bn\cdot D^2\delta\cdot\bn\;\; on\;\;\pOmega.
\end{align*}
Finally, for (d) we note that for $\bu=\varphi\bn +\bv$ with $\varphi\in\cD(\ol{\Omega})$ 
and $\bv\in\bcD({\Omega})$ we have 
\begin{align*}
\nabla_\bt\cdot\bu
=
\nabla\cdot\bu - \partial_\bn\bu\cdot\bn
&=
\nabla\cdot(\varphi\bn) 
-
\partial_\bn(\varphi\bn)\cdot\bn	
\\
&=
\partial_\bn\varphi + \varphi(\nabla\cdot\bn) 
- \partial_\bn\varphi - \varphi(\partial_\bn\bn\cdot\bn)
\\
&=
(\nabla_\bt\cdot\bn)(\bu\cdot\bn),
\end{align*}
which combined with \eqref{e:Dt*u=Dt*ut+H(u*n)}  shows that 
$\nabla_\bt\cdot\bu_\bt=0$ on $\partial\Omega$ for $\bu\in\bcD(\ol{\Omega})$,
and then by density~proves~(d) and \eqref{e:D*u=H(u*n)}.

Similarly for (e), we note that \eqref{e:<Dnu,phi>=<Dnu*n,phi*n>} holds for 
$\bu\in\bcD(\ol{\Omega})$ and $\bphi=\varphi\bn+\bv$, with $\varphi\in\cD(\ol{\Omega})$, 
$\bv\in\bcD(\Omega)$.
By density, it follows that it holds for 
$\bu\in\bD(\Delta,\bH^1(\Omega))$ and $\bphi\in\bH^1_\bn(\Omega)$.
\hfill$\Box$

\begin{lemma}\label{l:H1n, Poincare}
There exits $C_P$ depending only on $\Omega$~such~that 
\begin{equation}\label{e:H1bn,Poincare}
\|\bu\|^2_{\bL^2(\Omega)}
\leq 
C_P
\|\nabla \bu\|^2_{\bL^2(\Omega)},\quad \forall \bu\in \bH^1_\bn(\Omega),
\qquad\mbox{(Poincar\'e inequality)},
\end{equation}
where $|\nabla\bu|^2={\rm tr}(\nabla\bu\cdot{^t\nabla}\bu)$.
\end{lemma}
The lemma is proved easily by contradiction and do not present it here.
\\

\noindent
The result of the next lemma will be used in the lemma that follows it, where
the equivalence of two norms in $\bH^1_\bn(\Omega)$ is proved.
\begin{lemma}\label{l:div*div+curl*curl=grad*grad}
For every $\bu\in\bH^1_\bn(\Omega)$ we have
\begin{equation}\label{e:div^2+curl^2=grad*grad+H*(u*n)^2}
\int_\Omega 
\left(
|\nabla\cdot\bu|^2 + \frac{1}{2}|\nabla\times\bu|^2
\right)
=
\int_\Omega 
|\nabla\bu|^2
+
\int_{\pOmega}
(\nabla_\bt\cdot\bn)|\bu\cdot\bn|^2,
\end{equation}
where 
$|\nabla\bu|^2={\rm tr}(\nabla\bu\cdot{^t\nabla}\bu)$ and
$|\nabla\times\bu|^2={\rm tr}((\nabla\times\bu)\cdot{^t(}\nabla\times\bu))$.
\end{lemma}
{\bf Proof}.
In \cite[Theorem 3.1.1.1]{grisvard-1985} is proved  the following equality
\begin{equation}\label{e:grisvard,Thm3111}
\int_\Omega 
|\nabla\cdot\bu|^2 - {\rm tr}(\nabla\bu\cdot{\nabla}\bu)
=
\int_{\pOmega}
(\nabla_\bt\cdot\bn)|\bu\cdot\bn|^2,\quad \forall\bu\in\bH^1_\bn(\Omega).
\end{equation} 
We note that
\begin{align*}
|\nabla\bu|^2 - {\rm tr}(\nabla\bu\cdot{\nabla}\bu) 
&=
{\rm tr}(\nabla\bu\cdot{^t\nabla}\bu) - {\rm tr}(\nabla\bu\cdot{\nabla}\bu) 
\nonumber\\
&=
{\rm tr}((\nabla\bu-{^t\nabla}\bu)\cdot({^t\nabla}\bu -{\nabla}\bu))
+
{\rm tr}({^t\nabla}\bu\cdot({^t\nabla}\bu -{\nabla}\bu))
\nonumber\\
&=
|\nabla\times\bu|^2
+
{\rm tr}({^t\nabla}\bu\cdot{^t\nabla}\bu)
-
{\rm tr}({^t\nabla}\bu\cdot{\nabla}\bu)
\nonumber\\
&=
|\nabla\times\bu|^2
+
{\rm tr}({\nabla}\bu\cdot{\nabla}\bu)
-
|\nabla\bu|^2;\quad\mbox{hence}
\nonumber\\
- {\rm tr}(\nabla\bu\cdot{\nabla}\bu) 
&=
\frac{1}{2}|\nabla\times\bu|^2  - |\nabla\bu|^2 ,
\end{align*}
which combined with \eqref{e:grisvard,Thm3111}  proves \eqref{e:div^2+curl^2=grad*grad+H*(u*n)^2}.
\hfill$\Box$
\\

\begin{lemma}\label{l:|.|H1am~|.|H1}
Let $\alpha\geq0$, $\mu>0$ and define $\|\bu\|_{\bH^1_{\alpha,\mu}(\Omega)}$ by
\begin{equation}\label{e:H1amu}
\|\bu\|^2_{\bH^1_{\alpha,\mu}(\Omega)}
:=
\int_\Omega 
(\alpha|\bu|^2 +\mu|\nabla\bu|^2){d\bx}
+
\mu \int_{\pOmega}(\nabla_\bt\cdot\bn)(\bu\cdot\bn)^2 d{\boldsymbol\sigma}.
\end{equation}
Then 
\begin{align}
\|\bu\|^2_{\bH^1_{\alpha,\mu}(\Omega)}
&=
\int_\Omega 
\left(
\alpha|\bu|^2 +
\mu\left(|\nabla\cdot\bu|^2 + \frac{1}{2}|\nabla\times\bu|^2\right)
\right) {d\bx},
\label{e:H1am=H(u,div,rot)}	
\end{align}
and it is a norm in $\bH^1_\bn(\Omega)$ equivalent to the 
$\bH^1(\Omega)$ norm, i.e.
there exist $c_{\alpha,\mu}$, $C_{\alpha,\mu}$ positive constants depending only on $\alpha$, $\mu$ and $\Omega$ such that
\begin{equation}\label{e:|.|H1amu~|.|H1}
c_{\alpha,\mu}\|\bu\|^2_{\bH^1_{\alpha,\mu}(\Omega)}
\leq
\|\bu\|^2_{\bH^1(\Omega)}
\leq
C_{\alpha,\mu}\|\bu\|^2_{\bH^1_{\alpha,\mu}(\Omega)},\quad 
\forall \bu\in \bH^1_\bn(\Omega).
\end{equation}
\end{lemma}
{\bf Proof}.
The equality \eqref{e:H1am=H(u,div,rot)} follows from \eqref{e:H1amu} and 
\eqref{e:div^2+curl^2=grad*grad+H*(u*n)^2}.
Then from \eqref{e:H1am=H(u,div,rot)} we get
\begin{align}
\|\bu\|^2_{\bH^1_{\alpha,\mu}(\Omega)}
&\leq
\int_\Omega 
\left(\alpha|\bu|^2 + \mu C|\nabla\bu|^2\right){d\bx}
\leq
c_{\alpha,\mu}^{-1}
\|\bu\|^2_{\bH^1(\Omega)}\quad \mbox{with\ }\;\; c_{\alpha,\mu}^{-1}=\max\{\alpha,\mu C\},		\nonumber
\end{align}
which proves the left side of \eqref{e:|.|H1amu~|.|H1}.
Now we prove the right side of \eqref{e:|.|H1amu~|.|H1}.
From \cite{wahl-1992} we have
\begin{equation}\label{e:wahl}
\|\nabla\bu\|^2_{\bL^2(\Omega)}
\leq
C_{div,curl}
\left(
\|\nabla\cdot\bu\|^2_{\bL^2(\Omega)} + \frac{1}{2}\|\nabla\times\bu\|^2_{\bL^2(\Omega)}
\right),\quad \forall\bu\in\bH^1_\bn(\Omega),
\end{equation}
where $C_{div,curl}$ is a positive constant independent of $\bu$. Then
\begin{align*}
\|\bu\|^2_{\bH^1(\Omega)}
&=
\|\bu\|^2_{\bL^2(\Omega)}
+
\|\nabla\bu\|^2_{\bL^2(\Omega)}	\qquad\hfill&\mbox{(use  \eqref{e:H1bn,Poincare})}
\\
&\leq
(1+C_P)\|\nabla\bu\|^2_{\bL^2(\Omega)}\qquad&\mbox{(use \eqref{e:wahl})}
\\
&\leq
(1+C_P)C_{div,curl}
\left(
\|\nabla\cdot\bu\|^2_{\bL^2(\Omega)} + \frac{1}{2}\|\nabla\times\bu\|^2_{\bL^2(\Omega)}
\right)\qquad&
\\
&\leq
(1+C_P)\mu^{-1}C_{div,curl}
\left(
\alpha\|\bu\|^2_{\bL^2(\Omega)}
+
\|\nabla\cdot\bu\|^2_{\bL^2(\Omega)} + \frac{1}{2}\|\nabla\times\bu\|^2_{\bL^2(\Omega)}
\right)\qquad& \mbox{(use   \eqref{e:H1am=H(u,div,rot)})}
\\
&=
C_{\alpha,\mu}\|\bu\|^2_{\bH^1_{\alpha.\mu}(\Omega)},
\quad\mbox{with }\; C_{\alpha,\mu}=(1+C_P)\mu^{-1}C_{div,curl},
\end{align*}
which completes the proof.
\hfill$\Box$

\begin{lemma}\label{l:Dq}
The map 
$p\in\cD(\overline{\Omega})
\mapsto 
(\nabla p,p)\in\bcD(\overline{\Omega}))\times\cD(\partial\Omega)$ extends 
continuously from $D(\Delta,L^2(\Omega))$ to 
$\bH^{-1}_\bn(\Omega)\times H^{-1/2}(\partial\Omega)$.
In particular, for every $p\in D(\Delta,L^2(\Omega))$ and $\bv\in\bH^1_\bn(\Omega)$ we have
\begin{equation}\label{e:<Dp,v>}
\duality{\nabla p,\bv}
=\duality{p,\bv\cdot{\bf n}}
- \int_\Omega p(\nabla\cdot \bv){d\bx}.
\end{equation}
\end{lemma}
{\bf Proof}.
The claim for the extension of the map 
$p\in\cD(\overline{\Omega})
\mapsto 
p\in\cD(\partial\Omega)$ follows from (b.1), Lemma \ref{l:D(A;L2,H1)}.
For the claim for $\nabla p$ we note that for $p\in\cD(\overline{\Omega})$, 
$\bv\in\bH^1_\bn(\Omega)$ we have
\begin{subequations}\label{e:<Dq,v>~}
\begin{align*}
\duality{\nabla p,\bv}
&=
\duality{p,\bv\cdot{\bf n}}
- \int_\Omega p(\nabla\cdot \bv){d\bx},
\\
|\duality{\nabla p,\bv}|
&\leq 
C
\|p\|_{D(\Delta,L^2(\Omega))}
\|\bv\|_{\bH^1(\Omega)},
\end{align*}
\end{subequations}
which combined with the density of $\cD(\ol{\Omega})$ in $D(\Delta,L^2(\Omega))$ proves the claim for $\nabla p$ and \eqref{e:<Dp,v>}.
\hfill$\Box$


\noindent
\begin{lemma}\label{l:collection}
We have :
\begin{enumerate}
\item[(a)]
There exist $c_\bn$ and $C_\bn>0$, positive numbers,  such that 
\begin{equation}\label{e:uqh-H1/2~}
c_\bn \|\bu\|^2_{H^{1/2}(\partial\Omega)}
\leq 
\|\bu\cdot\bn\|^2_{H^{1/2}(\partial\Omega)}
\leq C_\bn \|\bu\|^2_{H^{1/2}(\partial\Omega)},
\quad\forall \bu\in\bH^1_\bn(\Omega).
\end{equation}
\item[(b)]
There exists $C_{div}>0$ such that for every  $q\in L^2(\Omega)$ with
$\int_\Omega q\,d\bx=0$, there exists  
$\bu_{div}(q)\in\bH^1_0(\Omega)$ such that
\begin{equation}\label{e:div(v)=q,v=g}
\nabla\cdot \bu_{div}(q)=q\;\; in\;\;\Omega,
\quad
\|\bu_{div}(q)\|^2_{\bH^1(\Omega)}\leq 
C_{div}
\|q\|^2_{L^2(\Omega)}.
\end{equation}
\end{enumerate}
\end{lemma}
{\bf Proof}.
The claim (a) 
follows easily by using the definition of the $H^{1/2}$ norm and the 
$C^{0,1}$ regularity  of $\bn$.
For the claim~(c),~see~\cite[Lemma 2.4]{temam-2001}.~\hfill$\Box$
\\

The following results are classical, see for example \cite{mclean-2000}, and are needed for 
writing the (conjugate) gradient method in $H^{-1/2}_0(\pOmega)$ space, see Section \ref{sec:alg}.
\begin{lemma}\label{l:DtN,NtD}
Let $D_N:H^{1/2}_0(\pOmega)\mapsto H^{-1/2}_0(\pOmega)$ and
    $N_D:H^{-1/2}_0(\pOmega)\mapsto H^{1/2}(\pOmega)$ be defined by 
\begin{subequations}\label{e:DtN,D,NtD,N}
\begin{align}
\hspace*{-7mm}
D_N(u)&=\partial_\bn D(u),
& where&&
D(u)\in H^1(\Omega),
\;\;
-\Delta D(u) &= 0\;\; \mbox{in}\;\; \Omega,\quad D(u)=u\;\; \mbox{on}\;\; \pOmega,
\label{e:D}\\
\hspace*{-7mm}
N_D(p)&=N(p),
&where&&
N(p)\in H^1(\Omega),
\;\;
-\Delta N(p) &= 0\;\; \mbox{in}\;\; \Omega,\quad \partial_\bn N(p)=p\;\; \mbox{on}\;\; \pOmega.
\label{e:N}
\end{align}
\end{subequations}
The maps\footnote{Commonly referred respectively as ``Dirichlet to Neumann"  and ``Neumann to Dirichlet" map} $D_N$ and $N_D$ are isomorphisms and $D_N=N_D^{-1}$.
Furthermore
\begin{subequations}\label{e:(.,.)H1/2,(.,.)H-1/2}
\begin{align}
(u,v)_{H^{1/2}_0(\pOmega)}
&\sim\duality{D_N(u),v}=(\nabla D(u),\nabla D(v))_{L^2(\Omega)},  
&&\forall u,v\in H^{1/2}_0(\pOmega),
\\
(p,q)_{H^{-1/2}_0(\pOmega)}
&\sim\duality{p,N_D(q)}=(\nabla N(p),\nabla N(q))_{L^2(\Omega)},
&&\forall p,q\in H^{-1/2}_0(\pOmega).
\end{align}
\end{subequations}
\end{lemma}

\section{Proof of the  main results}\label{sec:proof}
\begin{theorem}{\bf (proof of (a), Theorem \ref{th:main})}\label{th:uf-weak}
The problem \eqref{e:uf-all} has a unique weak solution ${\bomega}\in\bD(\Delta,\bH^1(\Omega))$
with $\nabla\cdot{\bomega}=0$.
\end{theorem}
{\bf Proof}.
Without loss of generality, we  assume that $\bg$ is the boundary trace of a 
$\bH^1(\Omega)$ satisfying
\begin{equation}\label{e:g}
\bg\in\bH^1(\Omega),\;\; \alpha\bg - \mu\Delta{\bg} = 0.
\end{equation}
Then $\bg\in \bD(\Delta,\bH^1(\Omega))$, and therefore from Lemma \ref{l:div(u),H1L2} we have
$\nabla\cdot\bg,\, \partial_\bn\bg\cdot\bn\in H^{-1/2}(\partial\Omega)$.
We  look for ${\bomega}$ in the form $\bomega={\bw}+\bg$. Then necessarily
$\bw$ satisfies
\begin{subequations}\label{e:bw-all}
\begin{align}
\alpha{\bw}-\mu\Delta{\bw}&=\bvarpi\;\; in\;\; \Omega,	\label{e:bw-Om}
\\
{\bw}_\bt&=0\;\; on\;\;\pOmega,						\label{e:bwt-pOm}
\\
\nabla\cdot\bw 
&=-
\nabla\cdot\bg
\;\; on\;\; \partial\Omega.	\label{e:dnbwn-pOm}
\end{align}
\end{subequations}
Taking into account \eqref{e:D*u=H(u*n)}, the condition \eqref{e:dnbwn-pOm} is equivalently written as
\begin{equation}\label{e:Dw*n}
\partial_\bn{\bw}\cdot\bn
=
-
(\nabla_\bt\cdot\bn)(\bw\cdot\bn)-\nabla\cdot\bg,
\end{equation}
which implies the following weak formulation for \eqref{e:bw-all}:
find ${\bw}\in\bH^1_\bn(\Omega)$ such that
\begin{equation}\label{e:bw-weak}
({\bw},{\bphi})_{\bH^1_{\alpha,\mu}(\Omega)}
= 
\int_\Omega 
(\bvarpi\cdot{\bphi}){d\bx}
-
\mu\duality{\nabla\cdot{\bg},\bphi\cdot\bn},
\quad 
\forall \bphi\in \bH^1_\bn(\Omega).
\end{equation}
From  Lemma \ref{l:|.|H1am~|.|H1}, the left hand side of \eqref{e:bw-weak} is a bilinear symmetric coercive form in  $\bH^1_\bn(\Omega)$, 
and from Lemma \ref{l:Dq} the  right hand side of \eqref{e:bw-weak} is in $\bH^{-1}_\bn(\Omega)$.
Therefore, \eqref{e:bw-weak} has a unique solution ${\bw}\in\bH^1_\bn(\Omega)$.

Taking $\bphi\in \bcD(\Omega)$ in \eqref{e:bw-weak}  shows that 
$\alpha{\bw}-\mu\Delta{\bw}=\bvarpi$ in $\bcD'(\Omega)$, 
which proves \eqref{e:bw-Om}, so \eqref{e:uf-Om}.
Clearly  \eqref{e:uft-pOm} holds from the definition of $\bw$.
To prove \eqref{e:dnufn-pOm} we note that \eqref{e:bw-Om} implies 
$\bw\in \bD(\Delta,\bH^1(\Omega))$ because $\bvarpi\in \bL^2(\Omega)$.
Then (a), Lemma \ref{l:D(A;L2,H1)} implies that $\partial_\bn\bw\in\bH^{-1/2}(\partial\Omega)$ and 
the following  formula of integration  by parts holds
\[
\duality{\partial_\bn\bw,\bphi}=
\int_\Omega(\bphi\cdot\Delta\bw +\nabla\bw:\nabla\bphi),\quad \bphi\in\bH^1_\bn(\Omega),
\]
which after using $\mu\Delta{\bw}=\alpha{\bw}-\bvarpi$  gives
\begin{equation}\label{e:Dw*n=}
\mu\duality{\partial_\bn\bw,\bphi}
=
\int_\Omega(\alpha(\bw\cdot\bphi)+\mu(\nabla\bw:\nabla\bphi) - \bvarpi\cdot\bphi,
\quad \bphi\in\bH^1_\bn(\Omega).
\end{equation}
By using \eqref{e:<Dnu,phi>=<Dnu*n,phi*n>} and \eqref{e:bw-weak}, from \eqref{e:Dw*n=} we get
\begin{eqnarray*}
\mu\duality{\partial_\bn\bw\cdot\bn,\bphi\cdot\bn}
&=&
\int_\Omega(\alpha(\bw\cdot\bphi)+\mu(\nabla\bw:\nabla\bphi)-\bvarpi\cdot\bphi){d\bx}
\\
&=&
-\mu\duality{(\nabla_\bt\cdot\bn)(\bw\cdot\bn)+\nabla\cdot{\bg},\bphi\cdot\bn};
\qquad\mbox{so}
\\
\partial_\bn\bw\cdot\bn + (\nabla_\bt\cdot\bn)(\bw\cdot\bn)
&=&
-\nabla\cdot{\bg}\quad in\; H^{-1/2}(\partial\Omega),
\end{eqnarray*}
which in view of  \eqref{e:D*u=H(u*n)} is nothing but 
\eqref{e:dnbwn-pOm}, or equivalently \eqref{e:dnufn-pOm}.

Now we show that $\nabla\cdot\bomega=0$ in $L^2(\Omega)$. We note that from \eqref{e:uf-Om} we get
\[
\nabla\cdot{\bomega}\in L^2(\Omega),\quad
\alpha(\nabla\cdot{\bomega})-\mu\Delta(\nabla\cdot{\bomega})=0\;\; 
\mbox{in}\;\; L^2(\Omega),
\]
so $\nabla\cdot{\bomega}\in D(\Delta,L^2(\Omega))$,
which combined with 
\eqref{e:dnufn-pOm},
implies  $\nabla\cdot\bomega=0$ in $L^2(\Omega)$ and completes~the~proof.
\hfill$\Box$

\begin{theorem}{\bf (proof of (b), Theorem \ref{th:main})}\label{th:uq-weak}
For  $q\in H^{-1/2}_0(\pOmega)$ let us still denote by $q$ its 
$D(\Delta,L^2(\Omega))$ harmonic extension. 
Then the problem \eqref{e:uq-all} has a unique weak 
solution $\bthetaup(q)\in\bH^1_\bn(\Omega)$ satisfying
$\nabla\cdot\bthetaup(q)=0$ in $L^2(\Omega)$. 
\end{theorem}
{\bf Proof}.
The proof follows the lines of the proof of Theorem \ref{th:uf-weak}.
Indeed, using Lemma \ref{l:Dq} and Lemma \ref{l:div(u),H1L2},
the weak formulation of  \eqref{e:uq-all} is:
find $\bthetaup(q)\in\bH^1_\bn(\Omega)$ such that
\begin{equation}\label{e:uq-weak}
(\bthetaup(q),\bphi)_{\bH^1_{\alpha,\mu}(\Omega)}
= 
-\duality{q,\bphi\cdot\bn} + \int_\Omega q(\nabla\cdot\bphi) {d\bx},
\quad 
\forall \bphi\in \bH^1_\bn(\Omega).
\end{equation}
From  Lemma \ref{l:|.|H1am~|.|H1}, the left hand side of \eqref{e:uq-weak} is a bilinear symmetric coercive form in  $\bH^1_\bn(\Omega)$, 
and from Lemma \ref{l:Dq}  the  right hand side of \eqref{e:bw-weak} is in $\bH^{-1}_\bn(\Omega)$.
Therefore, \eqref{e:uq-weak} has a unique solution $\bthetaup(q)\in\bH^1_\bn(\Omega)$.

Taking $\bphi\in\bcD(\Omega)$ shows that $\bthetaup(q)$ satisfies \eqref{e:uq-Om}.
Furthermore, $\bthetaup(q)$ satisfies \eqref{e:uqt-pOm} by the construction of $\bH^1_\bn(\Omega)$.
To prove that $\bthetaup(q)$ satisfies \eqref{e:dnuqn-pOm} we cannot proceed 
as in the proof of \eqref{e:dnufn-pOm} in Theorem \ref{th:uf-weak} because in general, 
$\nabla q\notin \bL^2(\Omega)$ and therefore $\bthetaup(q)\notin\bD(\Delta,\bH^1(\Omega))$.
Then we proceed as follows.
First we note that \eqref{e:uq-Om} and $\Delta q=0$ imply
\begin{equation}\label{e:div(uq)=0}
\nabla\cdot\bthetaup(q)\in L^2(\Omega),\quad
\alpha(\nabla\cdot\bthetaup(q))-\mu\Delta(\nabla\cdot\bthetaup(q))=0\;\; 
\mbox{in}\;\; L^2(\Omega).
\end{equation}
Next we take a sequence $(q^k)$ in $\bcD(\ol{\Omega})$ converging to $q$ in 
$D(\Delta,L^2(\Omega))$.
Further, we  consider the problem \eqref{e:uq-weak}, with $q^k$ instead of $q$, and denote its solution 
by $\bthetaup(q^k)$.
Then as in Theorem \ref{th:uf-weak} we get $\nabla\cdot\bthetaup(q^k)=0$
in $H^{-1/2}(\partial\Omega)$.
Taking the divergence of \eqref{e:uq-Om} with $(\bthetaup(q^k),q^k)$ instead of 
$(\bthetaup(q),q)$ gives
\begin{equation}\label{e:div(uqk)=0}
\nabla\cdot\bthetaup(q^k)\in L^2(\Omega),\quad
\alpha(\nabla\cdot\bthetaup(q^k))-\mu\Delta(\nabla\cdot\bthetaup(q^k))
=
-\Delta q^k\;\; 
\mbox{in}\;\; L^2(\Omega).
\end{equation}
Note that the convergence of $q^k$ to $q$ in $D(\Delta,L^2(\Omega))$ and
\eqref{e:uq-weak} with $q^k$ and $q$ imply 
$\lim_{k\to\infty}\bthetaup(q^k)=\bthetaup(q)$ in $\bH^1(\Omega)$, 
which combined with  \eqref{e:div(uq)=0} and \eqref{e:div(uqk)=0} implies
$\lim_{k\to\infty}\nabla\cdot\bthetaup(q^k)=\nabla\cdot\bthetaup(q)$ in 
$D(\Delta,L^2(\Omega))$.
Therefore 
$\lim_{k\to\infty}\nabla\cdot\bthetaup(q^k)=0=\nabla\cdot\bthetaup(q)$ in 
$H^{-1/2}(\partial\Omega)$, which proves \eqref{e:dnuqn-pOm}.

Finally, \eqref{e:div(uq)=0} and \eqref{e:dnuqn-pOm} imply $\nabla\cdot\bthetaup(q)=0$ in 
$L^2(\Omega)$, which completes the proof.
\hfill$\Box$

\begin{theorem}{\bf (proof of (c), Theorem \ref{th:main})}\label{th:A}
Let $A$ be the operator defined by
\begin{equation}
A:H^{-1/2}_0(\partial\Omega)\mapsto H^{1/2}_0(\partial\Omega),
\quad
Aq=-\bthetaup(q)\cdot\bn,
\label{e:A}
\end{equation}
where  $\bthetaup(q)$ is the solution of \eqref{e:uq-all}.
Then $A$ is is a self-adjoint, definite positive isomorphism from 
$H^{-1/2}_0(\partial\Omega)$ to $H^{1/2}_0(\partial\Omega)$ satisfying \eqref{e:qAq=qAp} and 
\eqref{e:m<(Aq,q)<M}.
\end{theorem}
{\bf Proof}.
For every $q\in H^{-1/2}_0(\pOmega)$ we consider 
its $D(\Delta,L^2(\Omega))$ harmonic extension as in Lemma \ref{l:Dq} still denoted by $q$.
%
\\
{\it i) $A$ is self-adjoint and non-negative, i.e. it satisfies \eqref{e:qAq=qAp}.}
Indeed, let $p,q\in H^{-1/2}_0(\pOmega)$. 
Then from \eqref{e:uq-weak} with $\bphi=\bthetaup(p)$ we get 
\begin{eqnarray}
\duality{p,Aq}
=
(\bthetaup(p),\bthetaup(q))_{\bH^1_{\alpha,\mu}(\Omega)}
=
\duality{q,Ap},	\label{e:<p,Aq>=<q,Ap>}
\end{eqnarray}
because $\nabla\cdot\bthetaup(p)=\nabla\cdot\bthetaup(q)=0$,
which proves the claim.
\\
{\it ii) 
$A$ is bounded  and satisfies \eqref{e:m<(Aq,q)<M}}.
Indeed, we have
\begin{subequations}
\begin{align}
\|Aq\|^2_{H^{1/2}_0(\partial\Omega)}
&=
\|\bthetaup(q)\cdot\bn\|^2_{H^{1/2}(\partial\Omega)} 
&\mbox{(use \eqref{e:uqh-H1/2~})}
\nonumber\\
&\leq
C_\bn\|\bthetaup(q)\|^2_{\bH^{1/2}(\partial\Omega)}
&\mbox{(use the trace theorem)}
\nonumber\\
&\leq
C_\bn C_t\|{\bthetaup}(q)\|^2_{\bH^1(\Omega)}
&\mbox{(use  \eqref{e:|.|H1amu~|.|H1})}
\nonumber\\
&\leq
C_\bn C_t C_{\alpha,\mu}
\|\bthetaup(q)\|^2_{\bH^1_{\alpha,\mu}(\Omega)}
&\mbox{(use \eqref{e:<p,Aq>=<q,Ap>})}
\nonumber\\
&=
C_\bn C_t C_{\alpha,\mu}
\duality{q,Aq}
\nonumber\\
&\leq
C_\bn C_t C_{\alpha,\mu}
\|q\|_{H^{-1/2}_0(\Omega)}
\|Aq\|_{H^{1/2}_0(\partial\Omega)},\quad\mbox{which implies}
\nonumber\\
\|Aq\|_{H^{1/2}_0(\partial\Omega)}
&\leq 
M\|q\|_{H^{-1/2}_0(\Omega)}\quad\mbox{and}	\label{e:|Aq|<M|q|}
\\
\duality{q,Aq}
&\leq
M\|q\|^2_{H^{-1/2}_0(\Omega)},\quad M=C_\bn C_t C_{\alpha,\mu}.		\label{e:<q,Aq><M|q|2}
\end{align}
\end{subequations}
Then \eqref{e:|Aq|<M|q|} shows that $A$ is bounded, and \eqref{e:<q,Aq><M|q|2} 
proves the right side inequality of \eqref{e:m<(Aq,q)<M}.

For the left side inequality of \eqref{e:m<(Aq,q)<M}, for given 
$q\in H^{-1/2}_0(\pOmega)$ we consider 
$\bthetaup_{div}(q)\in\bH^1_0(\Omega)$ as in (c), Lemma \ref{l:collection}.
Then  we get
\begin{subequations}
\begin{align}
&\hspace*{6mm}
\|q\|^2_{L^2(\Omega)}
&\mbox{(use \eqref{e:div(v)=q,v=g} and \eqref{e:uq-weak})}
\nonumber\\
&=
\int_\Omega q(\nabla\cdot {\bthetaup}_{div}(q))
=
(\bthetaup(q),{\bthetaup}_{div}(q))_{\alpha,\mu}
\nonumber\\
&\leq
\|\bthetaup(q)\|_{\alpha,\mu}\|{\bthetaup}_{div}(q)\|_{\alpha,\mu}
&\mbox{(use \eqref{e:|.|H1amu~|.|H1})}
\nonumber\\
&\leq
c^{-1/2}_{\alpha,\mu}
\|\bthetaup(q)\|_{\alpha,\mu}\|{\bthetaup}_{div}(q)\|_{\bH^1(\Omega)}
&\mbox{(use \eqref{e:div(v)=q,v=g})}
\nonumber\\
&\leq
c^{-1/2}_{\alpha,\mu}C^{1/2}_{div}
\|\bthetaup(q)\|_{\alpha,\mu}\|q\|_{L^2(\Omega)};\quad\mbox{hence}
\nonumber\\
\|q\|^2_{L^2(\Omega)}
&\leq
c^{-1}_{\alpha,\mu}C_{div}
\|\bthetaup(q)\|^2_{\alpha,\mu} 
&\mbox{(use \eqref{e:<p,Aq>=<q,Ap>})}
\nonumber\\
&=
c^{-1}_{\alpha,\mu}C_{div}
\duality{q,Aq}.		\label{e:|q|2<C<q,Aq>}
\end{align}
\end{subequations}
By combining \eqref{e:|q|H-1/2<C|q|L^2(Delta)} with \eqref{e:|q|2<C<q,Aq>}  gives
\begin{eqnarray}
\|q\|^2_{H^{-1/2}_0(\partial\Omega)}
&\leq&
c^{-1}_{\alpha,\mu}C_{div}C(\Omega)
\duality{q,Aq},\;\;\mbox{which implies}
\nonumber\\
m\|q\|^2_{H^{-1/2}(\partial\Omega)/\R}
&\leq&
\duality{q,Aq},\;\;\;\mbox{with}\;\;
m=\frac{c_{\alpha,\mu}}{C_{div}C(\Omega)},	\label{e:m|q|^2<{q,Aq>}}
\end{eqnarray}
and proves left side of \eqref{e:m<(Aq,q)<M}.
\\
{\it iii) The operator $A$ is an isomorphism}. 
Clearly,  \eqref{e:|Aq|<M|q|} shows that $A$ is continuous and 
\eqref{e:m|q|^2<{q,Aq>}} shows that $A$ is one-to-one.
Now, we show that $A$ is onto.
As $A$ is continuous and satisfies \eqref{e:m<(Aq,q)<M} we get that  $Im(A)$ is closed in $H^{1/2}_0(\partial\Omega)$.
Then  it is enough to show that $Im(A)$ is dense in $H^{1/2}_0(\partial\Omega)$.
From \cite[Corollary 1.8]{brezis-2011} this is reduced to prove the following:
assume that a certain $p\in H^{-1/2}_0(\partial\Omega)$ satisfies $\duality{p,Aq}=0$ for all 
$q\in H^{-1/2}_0(\pOmega)$, and then prove that $p=0$ in 
$H^{-1/2}_0(\partial\Omega)$.
Indeed, as $A$ is self-adjoint we get
$\duality{q,Ap}=0$, which implies $Ap=0$ in $H^{1/2}_0(\partial\Omega)$,
or equivalently $\bthetaup(p)\cdot\bn=0$. 
Then \eqref{e:uq-weak} and Lemma \ref{l:|.|H1am~|.|H1} imply $\bthetaup(p)=0$. 
Taking $\bphi\in\bcD(\Omega)$ in \eqref{e:uq-weak} implies 
$\nabla p=0$ in $\bcD'(\Omega)$, so $p$ is constant, or $p=0$ in 
$H^{-1/2}_0(\partial\Omega)$.
\hfill$\Box$

\begin{theorem}{\bf (proof of (d), Theorem \ref{th:main})}\label{th:u=v+omega,Au(q)=wn+g}
The problem \eqref{e:bu-all} has a solution 
$(\bu,p)\in\bH^1(\Omega)\times L^2(\Omega;\Delta)$, where 
$\bu= {\bomega} + \bthetaup(q)$ is unique,
$p=\pi + q$ is unique up to a constant, 
$q$ is the unique solution  in $H^{-1/2}_0(\partial\Omega)$  of
\begin{equation}\label{e:Au(q)=wn+gn}
Aq=({\bomega}-\bg)\cdot\bn),\;\; \mbox{or equivalently}\;\; 
-\bthetaup(q)\cdot\bn=({\bomega}-\bg)\cdot\bn.
\end{equation}
\end{theorem}
{\bf Proof}.
The existence of a solution $(\bu,p)$ of \eqref{e:bu-all} is a well-known fact,
\cite{girault+raviart-1986}.
Here a provide a different proof.
Adding the equations \eqref{e:uf-Om} and \eqref{e:uq-Om},
\eqref{e:uft-pOm} and  \eqref{e:uqt-pOm}  combined with \eqref{e:Au(q)=wn+gn},
gives
\begin{align*}
\alpha(\bomega+\bthetaup(q)) - \mu \Delta(\bomega+\bthetaup(q))
+\nabla (\pi + q)
&=\bvarpi + \nabla\pi 
= {\bf f} \;\; in\;\;\Omega,
\\
(\bomega+\bthetaup(q))_\bt
&=\bg_\bt\;\; on\;\; \partial\Omega,
\\
(\bomega+\bthetaup(q))\cdot\bn
&=\bg\cdot\bn\;\; on\;\; \partial\Omega,
\end{align*}
which shows that $\bu=\bomega+\bthetaup(q)$ satisfies \eqref{e:bu-Om} and \eqref{e:bu-pOm}.
Furthermore, $\bu=\bomega+\bthetaup(q)$ satisfies \eqref{e:div(bu)-Om} because 
$\nabla\cdot\bomega=0$ and $\nabla\cdot\bthetaup(q)=0$ in $L^2(\Omega)$.
\hfill$\Box$

\begin{remark}\label{r:u=thetaup+omega}
Theorem \ref{th:u=v+omega,Au(q)=wn+g} shows that the solution $\bu$ of the generalized 
Stokes problem \eqref{e:bu-all} is given as a superposition of  two divergence-free 
vector fields, $\bomega$ and  $\bthetaup(q)$.
Both  $\bomega$ and $\bthetaup(q)$ solve classical elliptic screened Poisson equations, 
respectively \eqref{e:uf-all} and \eqref{e:uq-all}.
Contrary to what usually is done, here we do not add a domain constraint or a penalty term  
to count for the divergence-free condition. Instead we impose some appropriate boundary conditions and use the structure of the screened Poisson equation.

More precisely, the decomposition  $\bu=\bomega+\bthetaup(q)$ is driven by the Helmholtz 
decomposition \eqref{e:f=Dp+w} of the external force, $\ibf=\bvarpi + \nabla\pi$.
The rotational ``force" term $\bvarpi$ generates the divergence-free velocity
$\bomega$, while the gradient ``solid pressure" force term $-\nabla q$ 
generates the divergence-free velocity $\bthetaup(q)$.
The divergence-free of both velocities $\bomega$ and $\bthetaup(q)$ follows from 
the structure of the screened Poisson equation, divergence-free nature of $\bvarpi$ and $-\nabla q$,
and the appropriate boundary conditions for $\bomega$ and $\bthetaup(q)$.

Regardless the ``solid boundary pressure" term $q$,
$\bu=\bomega+\bthetaup(q)$ solves \eqref{e:bu-Om}, \eqref{e:div(bu)-Om} and the tangential component of \eqref{e:bu-pOm}. The role of $q$ is to simply adjust the field 
$\bthetaup(q)$ so that $\bu=\bomega+\bthetaup(q)$
satisfies even the normal component of\eqref{e:bu-pOm}.
%
\end{remark}

\section{New numerical methods for the generalized Stokes problem \eqref{e:bu-all}}\label{sec:alg}
Theorem \eqref{th:u=v+omega,Au(q)=wn+g} shows that we can solve \eqref{e:bu-all} equivalently by first solving \eqref{e:uf-all}, and next \eqref{e:uq-all} with $q$ the solution of \eqref{e:Au(q)=wn+gn}. For given $\bvarpi$ and $q$, both these problems are screened Poisson equations and can be solved efficiently by a variety of methods.
The equation \eqref{e:Au(q)=wn+gn} is then the most important element of the method.
We can approximate \eqref{e:Au(q)=wn+gn} directly of iteratively.

\subsection{Direct method}
Let  $(\varphi_i)$ be a Riesz basis of $H^{-1/2}_0(\pOmega)$, \cite[Chapter 3]{christensen-2016}. Then \eqref{e:Au(q)=wn+gn} is equivalently~formulated~as 
\begin{equation}\label{e:Aq=w-g,weak}
\duality{\varphi_i,Aq}=\duality{\varphi_i,(\bomega-\bg)\cdot\bn},\quad \forall i\in\N.
\end{equation}
We look for $q=\sum_{k=1}^\infty c_k\varphi_k$, where 
 $(c_k)\in\ell^{-1/2}_0$ and
\begin{equation}
\ell^{-1/2}_0 =
\left\{ (\alpha_k),\;\;
a\sum_{k=1}^\infty|\alpha_k|^2 
\leq
\left\|\sum_{k=1}^\infty \alpha_k\varphi_k\right\|^2_{H^{-1/2}(\pOmega)}
\leq 
b\sum_{k=1}^\infty|\alpha_k|^2,\quad
\sum_{k=1}^\infty \alpha_k\duality{\varphi_k,1}=0
\right\},
\end{equation}
with $a,b>0$  independent of $(\alpha_k)$. 
The convergence of the series $q=\sum_{k=1}^\infty c_k\varphi_k$ is in the sense that 
$\lim_{l\to\infty}(q-s_l(q))=0$ in $H^{-1/2}_0(\pOmega)$, where $s_l(q):=\sum_{k=1}^l c_k\varphi_k$.
Clearly the coefficients $(c_k)$ satisfy 
\begin{equation}\label{e:c^k;1,inf}
\sum_{k=1}^\infty 
c_k\duality{\varphi_i,A\varphi_k} 
=
\duality{\varphi_i,(\bomega-\bg)\cdot\bn},\quad \forall i\in\N.
\end{equation}
We look for an approximated solution of \eqref{e:c^k;1,inf}.
This approach is connected to the finite element approximation of \eqref{e:bu-all}, because every 
finite set of linearly independent elements of $H^{-1/2}_0(\pOmega)$, which in numerical 
applications is any finite element basis, can be completed to a Riesz basis.

We consider the finite version of \eqref{e:c^k;1,inf}.
For given $l\in\N$, we look for $q_l:=\sum_{k=1}^l c_{k,l}\varphi_{k}$ satisfying
\begin{equation}\label{e:c^k_l;1,l}
\sum_{k=1}^l 
c_{k,l} \duality{\varphi_i,A\varphi_k} 
=
\duality{\varphi_i,(\bomega-\bg)\cdot\bn},\quad \forall i=1,\ldots,l.
\end{equation}
%
%
%
\begin{theorem}\label{th:|pl-ql|->0}
For given $l\in\N$, the problem \eqref{e:c^k_l;1,l} has a unique solution 
$(c_{k,l})_{k=1}^l$ in 
$\ell^{-1/2}_{0,l}:=\ell^{-1/2}_{0}\cap\{(\alpha_k),\; \alpha_{l+1}=\cdots=0\}$.
Furthermore
\begin{subequations}\label{e:|*l-*|->0}
\begin{align}
\|q_l - q\|_{H^{-1/2}_0(\pOmega)}
&\leq
\frac{M}{m}\|q - s_l(q)\|_{H^{-1/2}_0(\pOmega)}, 		\label{e:|ql-q|->0}
\\
\|\bu_l - \bu\|_{H^{1}(\Omega)}
&\leq
\frac{C_{\alpha,\mu}^{1/2}M^{3/2}}{m}\|q - s_l(q)\|_{H^{-1/2}_0(\pOmega)},
\quad 
\mbox{where}\;\;\bu_l=\bomega+\bthetaup(q_l). 	\label{e:|ul-u|->0}
\end{align}
\end{subequations}
\end{theorem}
{\bf Proof}.
The existence and uniqueness of the solution to \eqref{e:c^k_l;1,l}  holds because
the matrix $[\duality{\varphi_i,A\varphi_k}]_{i,k=1}^l$
is symmetric definite-positive in $\ell^{-1/2}_{0,l}$.
Subtracting  \eqref{e:c^k;1,inf}  from \eqref{e:c^k_l;1,l} gives
$
\duality{\varphi_i,A(q_l-q)}
=
0$,
which after multiplying it by $c_{i,l}-c_i$ and summing the result for $i=1,\ldots,l$ gives
\begin{equation}\label{e:ckl-ck}
\duality{q_l-s_l(q),A(q_l-q)}
=
0,
\;\;\;
\mbox{or}\;\;\;
\duality{q_l-q,A(q_l-q)}
=
\duality{s_l(q)-q,A(q_l-q)}.
\end{equation}
Taking into account \eqref{e:m|q|^2<{q,Aq>}} and \eqref{e:|Aq|<M|q|},
from \eqref{e:ckl-ck} we get
\begin{eqnarray*}\label{e:m|pl-ql|^2<}
m\|q_l-q\|^2_{H^{-1/2}_0(\pOmega)}
&\leq&
\duality{q_l-q,A(q_l-q)}
=
\duality{s_l(q)-q,A(q_l-q)}
\\
&\leq&
M
\|q-s_l(q)\|_{H^{-1/2}_0(\pOmega)}
\|q_l-q\|_{H^{-1/2}_0(\pOmega)},
\end{eqnarray*}
which after simplifying by $\|q_l-q\|_{H^{-1/2}_0(\pOmega)}$ proves 
\eqref{e:|ql-q|->0}.

For \eqref{e:|ul-u|->0} we note that $\bu_l-\bu=\bthetaup(q_l)-\bthetaup(q)$, and then from
\eqref{e:|.|H1amu~|.|H1}, \eqref{e:uq-weak} with $q_l-q$ instead of $q$ and
with $\bphi=\bthetaup(q_l)-\bthetaup(q)$, and \eqref{e:<q,Aq><M|q|2} we get
\begin{align*}
\|\bu_l-\bu\|^2_{\bH^1(\Omega)}
&=
\|\bthetaup(q_l)-\bthetaup(q)\|^2_{\bH^1(\Omega)}
\leq
C_{\alpha,\mu}
\|\bthetaup(q_l)-\bthetaup(q)\|^2_{\bH^1_\bn(\Omega)}
=
C_{\alpha,\mu}
\duality{q_l-q,A(q_l-q)}
\\
&\leq
C_{\alpha,\mu}M
\|q_l-q\|^2_{H^{-1/2}(\pOmega)/\R},
\end{align*}
which combined with \eqref{e:|ql-q|->0} proves \eqref{e:|ul-u|->0}.
\hfill$\Box$

\begin{remark}\label{r:direct-method}
The estimates \eqref{e:|ql-q|->0} shows that the ``finite approximation" 
\eqref{e:c^k_l;1,l} converges to the solution of \eqref{e:c^k;1,inf} as $l\to\infty$.
For a given $l$, the bounds in \eqref{e:|*l-*|->0} depend on the choice of 
the basis $\{\varphi_1,\ldots,\varphi_l\}$.
In the finite element context, it depends on the approximation order 
of the finite element basis $\{\varphi_1,\ldots,\varphi_l\}$ in $H^{-1/2}_0(\pOmega)$.
\end{remark}

\subsection{Iterative method}\label{ss:iterative}
Now we show that we can solve  $Aq=(\bomega-\bg)\cdot\bn$ with a gradient, or conjugate gradient method. We introduce 
\begin{equation}\label{e:a(p,q),b(q)}
a(p,q)=\duality{p,Aq},
\quad
b(q)=\duality{q,(\bomega-\bg)\cdot\bn},
\quad
j(q)=\frac{1}{2}a(q,q)-b(q),
\quad
p,q\in H^{-1/2}_0(\partial\Omega).
\end{equation}
\begin{lemma}\label{l:e->min}
The followings hold.
\begin{enumerate}
\item[(a)]
The functional $j$ is differentiable in $H^{-1/2}_0(\pOmega)$ and 
its derivative $j'$ at $q$ in the direction $p$ is given by
\begin{equation}\label{e:j'(q)}
j'(q;p)
=
\duality{p,A(q)-(\bomega-\bg)\cdot\bn},\quad \forall p\in H^{-1/2}_0(\pOmega),
\end{equation}
\item[(b)]
Let the inner-product in $H^{-1/2}_0(\Omega)$ be given by 
$\duality{p,N_D(p)}$, see Lemma \ref{l:D,N-iso}.
Then
\begin{equation}\label{e:g(q)}
j'(q;p)=(g(q),p)_{H^{-1/2}_0(\Omega)},
\quad
g(q)=D_N(A(q)-(\bomega-\bg)\cdot\bn),
\quad \forall p\in H^{-1/2}_0(\pOmega),
\end{equation}
where $D_N$ is the operator defined in Lemma \ref{l:D,N-iso}.
\item[(c)]
Lastly, the functional $j$ achieves its minimum in $H^{-1/2}_0(\partial\Omega)$
at a unique $q$, which is the unique solution in 
$H^{-1/2}_0(\partial\Omega)$ of $Aq=(\bomega-\bg)\cdot\bn$.
\end{enumerate}
\end{lemma}
{\bf Proof}.
The differentiability of $j$ and the formula \eqref{e:j'(q)} follows by straightforward 
calculations.
The claim (b)  follows from Lemma \ref{l:D,N-iso}.
The claim (c) follows from Lax-Milgram lemma, \cite[Corollary 5.8]{brezis-2011},
because in view \eqref{e:m<(Aq,q)<M}, the bilinear form $a(p,q)$ is continuous and coercive
in $H^{-1/2}_0(\partial\Omega)$, and $b(q)$ is continuous in $H^{-1/2}_0(\partial\Omega)$.
\hfill$\Box$

\begin{proposition}\label{p:alg,convergence}
Let $(q_k)$ be the sequence associated to the conjugate gradient method applied to the functional 
$q\mapsto j(q)$, as given in \cite[Section 4.3.3]{bristeau+al-1987}.
Then $(q_k)$ converges to the solution $q$ of $Aq=(\bomega-\bg)\cdot\bn$ with convergence rate 
\begin{equation}
\|q_k-q\|_{H^{-1/2}_0(\pOmega))}
\leq 
C(q_0)\left(\frac{\sqrt{\kappa}-1}{\sqrt{\kappa}+1}\right)^k,\quad k=1,2,\ldots,
\end{equation}
where $C(q_0)>0$ and $\kappa$ is the 
condition number of the operator $D_N A$, with $1\leq\kappa\leq Mm^{-1}$
and $m$, $M$ given in (c), Theorem \ref{th:main}, and
$D_N$ is defined in Lemma \ref{l:D,N-iso}.
\end{proposition}
{\bf Proof}.
The theorem follows from \cite[Section 4.3.3]{bristeau+al-1987}, where we take:
\\
\begin{tabular}{p{195mm}}
$\bullet$\; 
$(V,(\cdot,\cdot)_V)=(H^{-1/2}_0(\pOmega), (p,q)_{H^{-1/2}_0(\pOmega)})$, with
$(p,q)_{H^{-1/2}_0(\pOmega)}
=\duality{p,N_D(q)}
=
(\nabla N(p),\nabla N(q))_{L^2(\Omega)}$,
\\
$\bullet$\; 
$a(p,q)=\duality{p,A q}=(p,D_N(A(q)))_{H^{-1/2}_0(\pOmega)}$, and
\\
$\bullet$\;
$L(p)=\duality{p,(\bw-\bg)\cdot\bn}$.
\end{tabular}

\noindent
We note that the coercivity of $D_NA$ follows from \eqref{e:m<(Aq,q)<M},
which furthermore shows that the spectrum of $D_NA$ is included in $[m,M]$
and implies $1\leq\kappa\leq Mm^{-1}$.
\hfill$\Box$

\section{Numerical results}\label{sec:num-results}
In an upcoming paper we will present a detailed numerical analysis and the numerical 
implementation of our method, with applications to generalized Stokes equations and unsteady Navier-Stokes equations in dimension two and three.
In this section we present some numerical results for some two and three dimensional classical benchmarks problems to demonstrate our method.
All the numerical results are obtained based on the iterative approach,
as described in Subsection \ref{ss:iterative}.
We have used \texttt{FreeFEM}, an open-source PDE solver based on finite element method, see \cite{hecht-2012}.
We have observed that by using the conjugate gradient algorithm as given in Proposition \ref{p:alg,convergence}, or the \texttt{LinearCG} function applied to the discrete functional
$j(q)$ gives very similar results.
 
In all the cases we have used a $(P_2,P_2)$ in 2d, or $(P_2,P_2,P_2)$ in 3d, finite element approximation for the velocity and a $P_2$ approximation for the pressure. 
The domain is always meshed by discretizing its boundary with  $n$ 
points, and next discretizing the domain in triangles/tetrahedrons.
The quantity $h=1/n$ refers to the mesh size of the discretization.

\subsection{The 2D Kovasznay test problem}
We consider the two-dimensional trigonometric Kovasznay problem given by \eqref{e:bu-all},
with
\begin{subequations}\label{e:2d-k}
\begin{align*}
\alpha&=0,\quad
\mu=1,\quad m\in\N,
\\
\ibf(x,y)&=
 \left(2m^2\cos(mx)\sin(my) + \frac{1}{2}m\sin(2mx),
 -2m^2\sin(mx)\cos(my) + \frac{1}{2}m\sin(2my)\right).
\end{align*}
\end{subequations}
The exact solution is given by
\[
 \bu(x,y)=\left(\cos(mx)\sin(my), -\sin(mx)\cos(my)\right),\quad
 p(x,y)= -\frac{1}{4}\left(\cos(2mx) + \cos(2my)\right).
\]
The approximated velocity, resp. pressure, will be denoted by $\bu_{n,m}$, resp.
pressure $p_{n,m}$.

\subsubsection*{The case of a disk}
\noindent
Here $\Omega$ is a disk with radius $\pi$, with $n=100$ discretization points on the boundary, which corresponds to $h=2\pi^2/100\approx 0.2$, and $m=1$.
In Figure \ref{f:k-disk-100,1} we plot (from left to right, top to bottom) 
the mesh,  $p_{n,m}$, $\bu_{n,m}$ and the isolines of $p_{n,m}$ for  $(n,m)=(100,1)$.
\begin{figure}[!h]
\noindent
\hspace*{-10mm}
\includegraphics[scale=0.7]{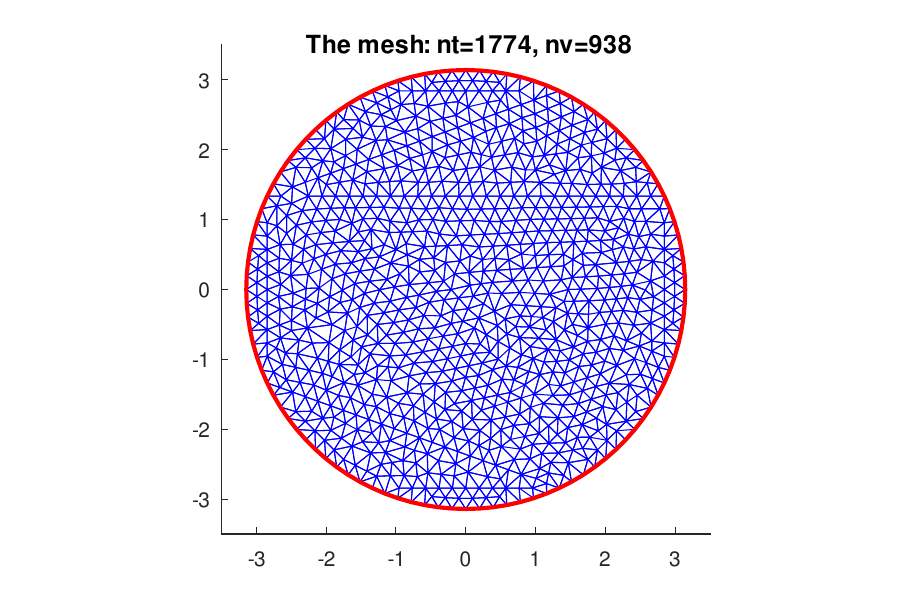} 
\hspace*{-25mm}
\includegraphics[scale=0.75]{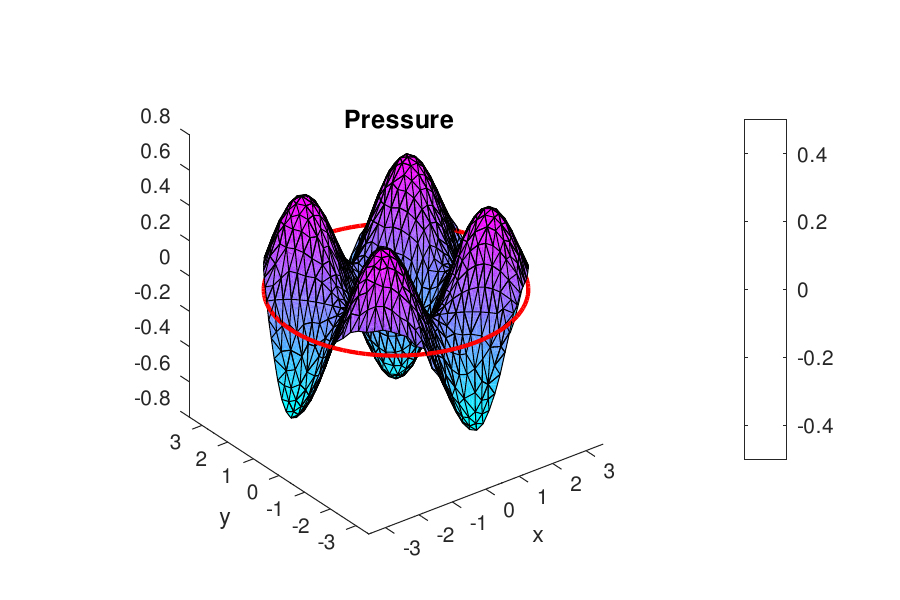} 
\\
\noindent
\hspace*{-28mm}
\includegraphics[scale=0.85]{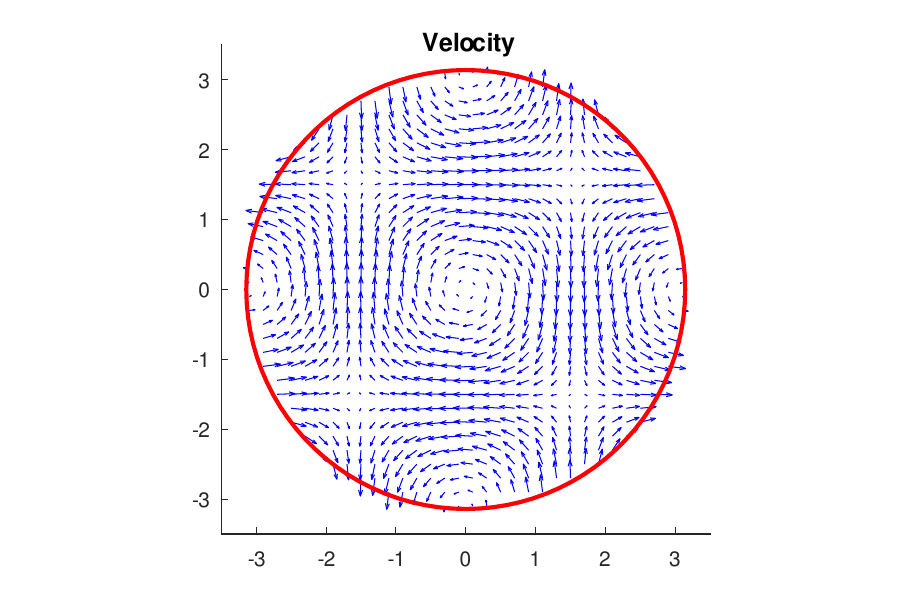} 
\hspace*{-30mm}
\includegraphics[scale=0.85]{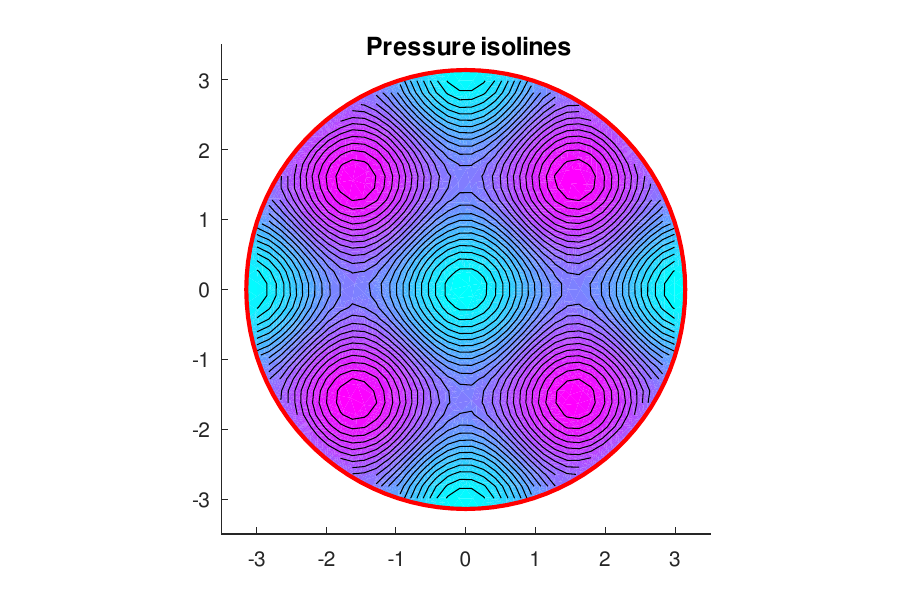} 
\caption{Results for Kovasznay problem with $\Omega=B(0,\pi)$, 
$n=100$ boundary points (so $h=0.19736$) and $m=1$. 
Here 
$\|\bu_{n,m}-\bu\|_{L^\infty(\Omega)}/\|\bu\|_{L^\infty(\Omega)}=1.0\cdot 10^{-3}$,
$\|p_{n,m}-p\|_{L^\infty(\Omega)}/\|p\|_{L^\infty(\Omega)}= 6.1\cdot10^{-2}$,
$\|\nabla\cdot\bu_{n,m}\|_{L^\infty(\Omega)}=1.9\cdot 10^{-2}$,
$\|\nabla\cdot\bu_{n,m}\|_{L^2(\Omega)}=6.2\cdot 10^{-3}$.
}
\label{f:k-disk-100,1}
\end{figure}

\subsubsection*{The case of a general smooth domain}
\noindent
Here $\Omega$ is a perturbation of the disk $B(0,\pi)$ with $\partial\Omega$ given by
\begin{equation}\label{e:T(B,pi)}
\begin{array}{rcl}
\partial\Omega
&=&\{(x(t),y(t)),\quad
\ \ x(t)=\pi\cos(t)(1+0.1\sin(t)\sin(t)),\\
&& \hspace*{29mm}
   y(t)=\pi \sin(t)(0.7+0.1\cos(4t)^3),\quad t\in[0,2\pi]\}.
\end{array}
\end{equation}
In Figure \ref{f:k-gen-300,2} we plot (from left to right, top to bottom) 
the mesh,  $p_{n,m}$, $\bu_{n,m}$ and the isolines of $p_{n,m}$ for $(n,m)=(300,2)$.

\begin{figure}[!h]
\noindent
\hspace*{-18mm}
\includegraphics[scale=0.75]{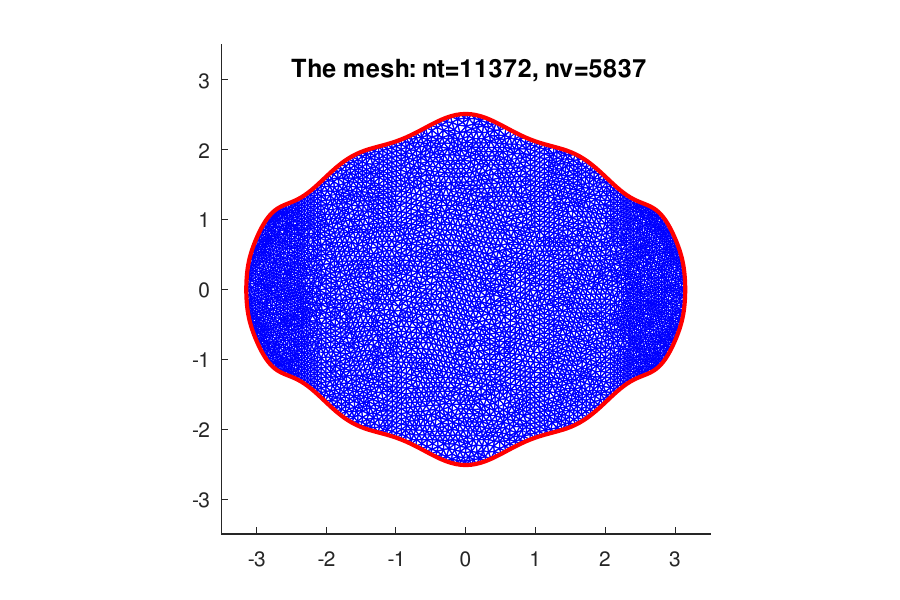} 
\hspace*{-28mm}
\includegraphics[scale=0.75]{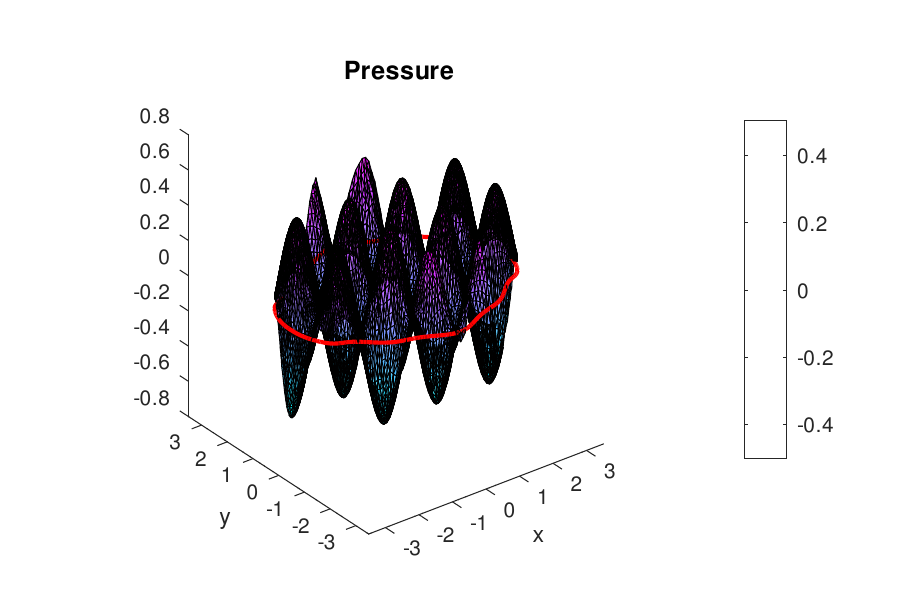} 
\\
\noindent
\hspace*{-30mm}
\includegraphics[scale=0.85]{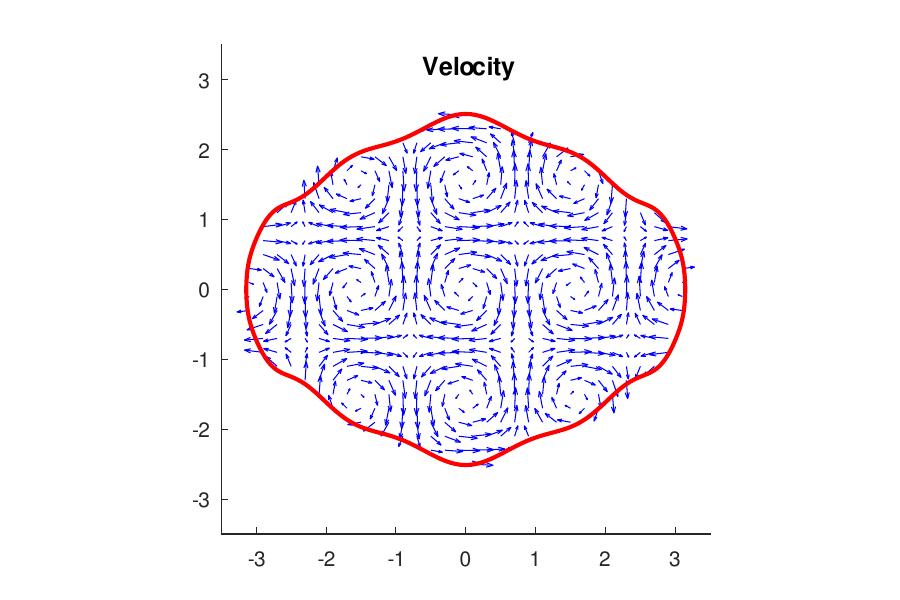} 
\hspace*{-31mm}
\includegraphics[scale=0.85]{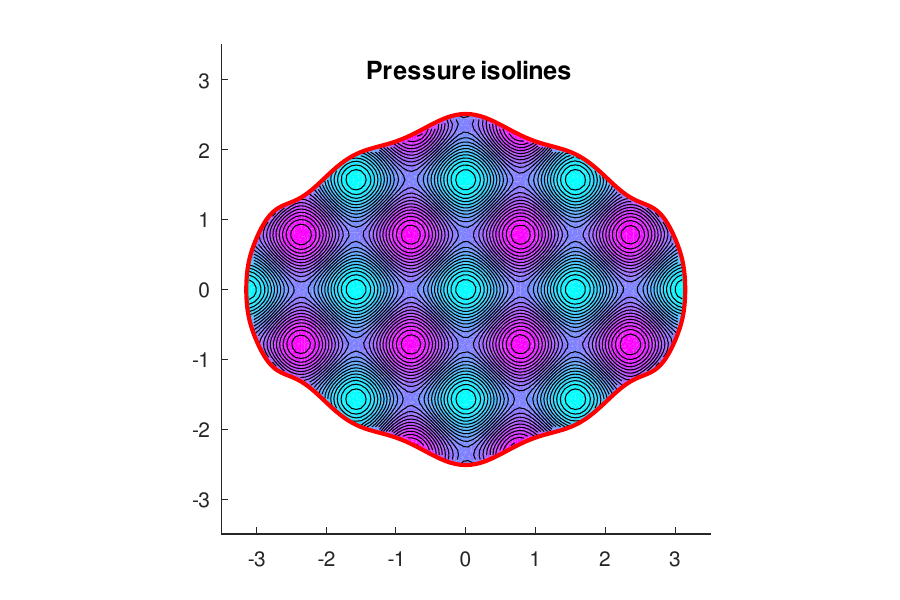} 
\caption{Results for Kovasznay problem with $\Omega$ given by \eqref{e:T(B,pi)}, 
$n=300$ boundary points (so $h=0.0694923$) and $m=2$. 
Here $\|\bu_{n,m}-\bu\|_{L^\infty(\Omega)}/\|\bu\|_{L^\infty(\Omega)}=2.2\cdot 10^{-3}$,
$\|p_{n,m}-p\|_{L^\infty(\Omega)}/\|p\|_{L^\infty(\Omega)}=7.8\cdot 10^{-2}$,
$\|\nabla\cdot\bu_{n,m}\|_{L^\infty(\Omega)}=3.4\cdot 10^{-2}$,
$\|\nabla\cdot\bu_{n,m}\|_{L^2(\Omega)}=1.0\cdot 10^{-2}$.
}
\label{f:k-gen-300,2}
\end{figure}

\subsection{The 2D Bercovier–Engelman test problem}
This example is chosen to demonstrate that the method gives very good results even in the case of 
domains which are not $C^{1,1}$.
In this case we consider the so-called  ``2D Bercovier–Engelman problem", see \cite{angeli+al-2017,boyer+al-2017}, given by \eqref{e:bu-all} with
\begin{subequations}\label{e:2d-b-e}
\begin{align*}
\alpha&=0,\quad
\mu=1,\quad
\Omega=(0,1)\times(0,1)
\\
\ibf(x,y)&=\left( f(x,y) + \left(y-\frac{1}{2}\right),
                 -f(y,x) + \left(x-\frac{1}{2}\right)\right),
                 \\
f(x,y)&=256\left( x^2 (x-1)^2(12y-6) + y(y-1)(2y-1)(12x^2-12x + 2) \right).
\end{align*}
\end{subequations}
The exact solution is given by
\[
 \bu=(u(x,y),-u(y,x),\quad
  u(x,y)=-256x^2 (x-1)^2y(y-1)(2y-1),\quad
  p(x,y)=\left(x-\frac{1}{2}\right)\left(y-\frac{1}{2}\right).
\]
One of the interests of this  classical test case is that the gradient part $\nabla\pi$ of the 
external force $\ibf$ is small compared to its curl part $\bvarpi$, and thus it is interesting to check whether a numerical method will capture the pressure correctly, see 
\cite{boyer+al-2017}. 
In Figure \ref{f:b-50} we plot (from left to right, top to bottom) the mesh,
the pressure $p_n$, velocity $\bu_n$ and the isolines of pressure $p_n$, with $n=50$.
\\
\begin{figure}[!h]
\noindent
\hspace*{0mm}
\includegraphics[scale=0.6]{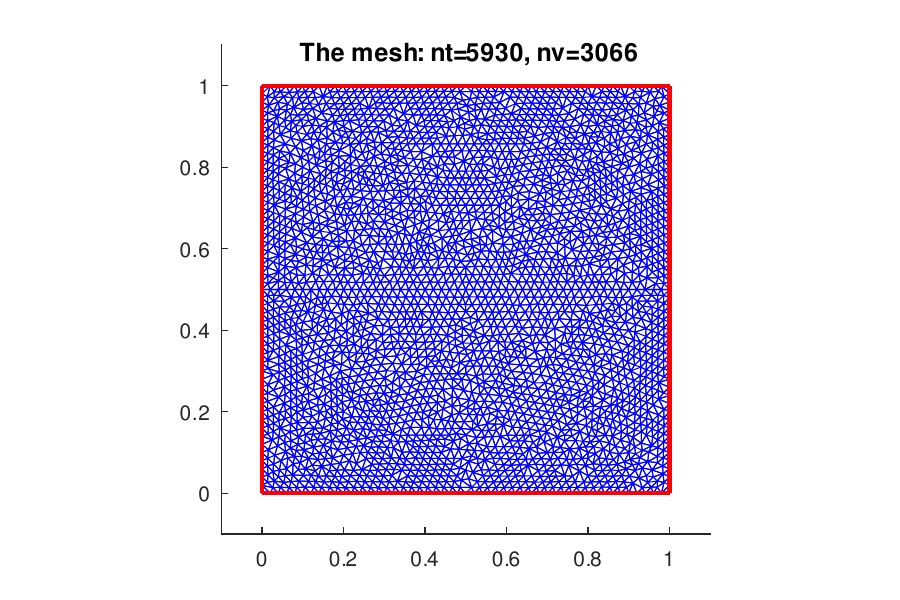} 
\hspace*{-5mm}
\includegraphics[scale=0.6]{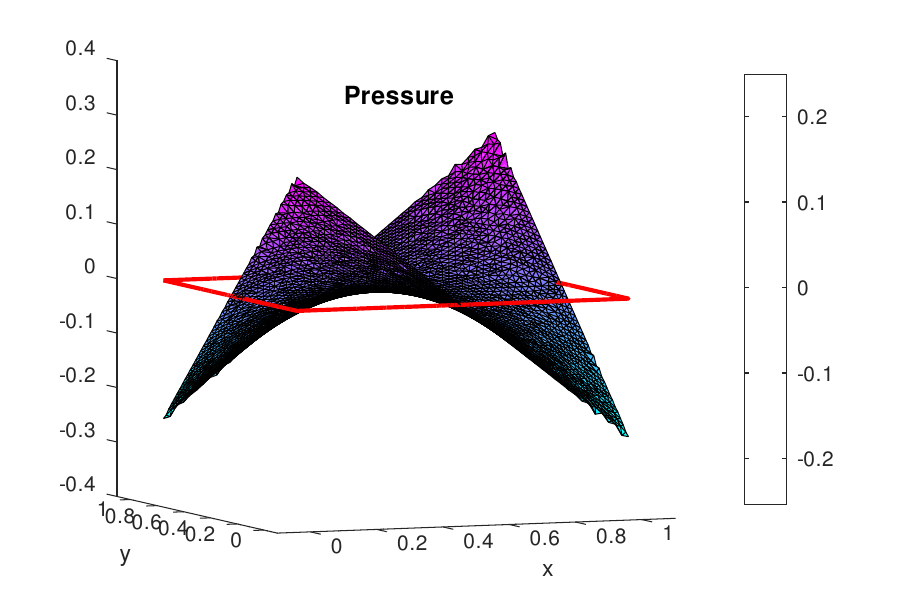} 
\\
\\
\noindent
\hspace*{-15mm}
\includegraphics[scale=0.8]{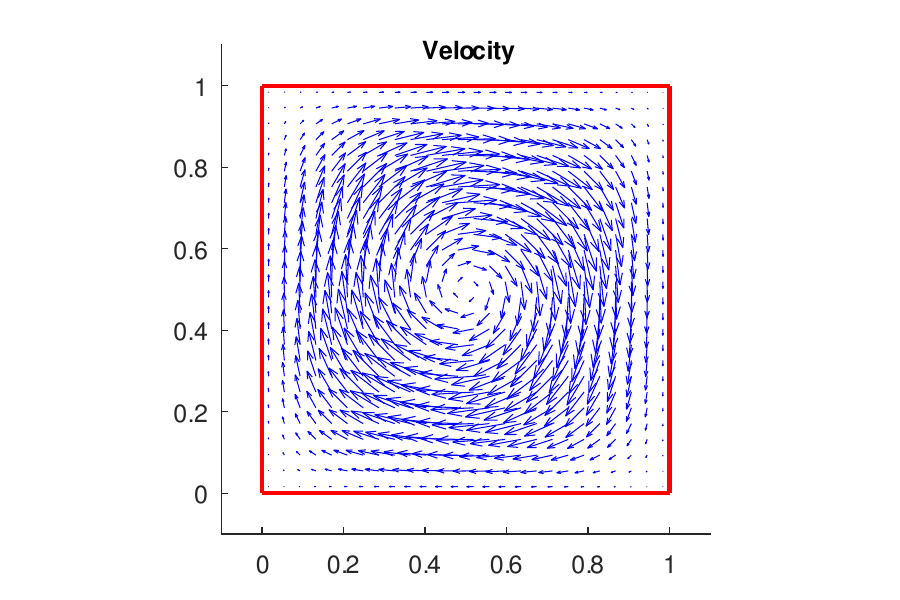} 
\hspace*{-32mm}
\includegraphics[scale=0.8]{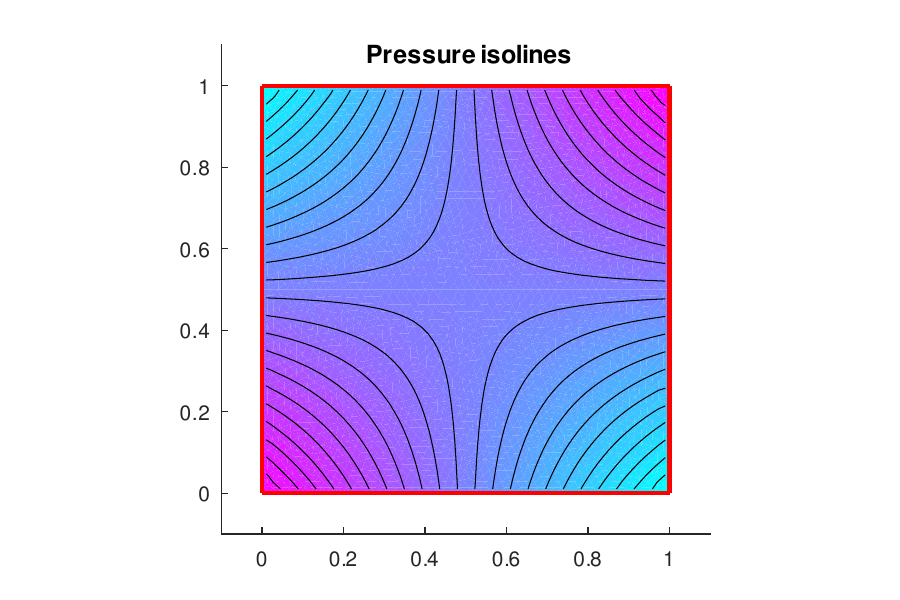} 
\caption{Results for Bercovier-Engleman with $n=50$  (so $h=0.02$). 
Here 
$\|\bu_n-\bu\|_{L^\infty(\Omega)}/\|\bu\|_{L^\infty(\Omega)}=2.7\cdot 10^{-5}$,
$\|p_n-p\|_{L^\infty(\Omega)}/\|p\|_{L^\infty(\Omega)}=5.8\cdot 10^{-2}$,
$\|\nabla\cdot\bu_n\|_{L^\infty(\Omega)}=2.9\cdot 10^{-2}$,
$\|\nabla\cdot\bu_n\|_{L^2(\Omega)}=3.8\cdot 10^{-3}$.
}
\label{f:b-50}
\end{figure}


\subsection{3d Taylor-Green vortex test problem}
This example is chosen to demonstrate that the method works very well with 
domains which are not $C^{1,1}$ in the case $d=3$.
In this case we consider the so-called  ``Taylor-Green vortex problem",
see \cite{angeli+al-2017,boyer+al-2017}, given by \eqref{e:bu-all} with
\begin{subequations}\label{e:3d-t-g}
\begin{align*}
\alpha&=0,\quad
\mu=1,\quad
\Omega=(0,1)\times(0,1)\times(0,1),
\\
\ibf(x,y)&=\left(-36\pi^2\cos(2\pi x_1)\sin(2\pi x_2)\sin(2\pi x_3), 0, 0\right).
\end{align*}
\end{subequations}
The exact solution is given by
\[
\bu=
\left[
\begin{array}{r}
 -2\cos(2\pi x_1)\sin(2\pi x_2)\sin(2\pi x_3)\\
 \sin(2\pi x_1)\cos(2\pi x_2)\sin(2\pi x_3)\\
 \sin(2\pi x_1)\sin(2\pi x_2)\cos(2\pi x_3)
 \end{array}
\right],
\hspace{10mm}
p =-6\pi\sin(2\pi x_1)\sin(2\pi x_2)\sin(2\pi x_3).
\]
The domain $\Omega$ is discretized with regular tetrahedrons with a parameter 
$\ds{h=\frac{1}{n}}$. 
In Figure \ref{f:tg-40} we plot (from left to right, top to bottom) the approximate velocity  
$\bu_n$ in $xz$, $yz$, $xy$ planes and the pressure $p_n$ in different cross-sections,
with $n=40$.
\begin{figure}[!h]
\noindent
\hspace*{-22mm}
\includegraphics[scale=0.8]{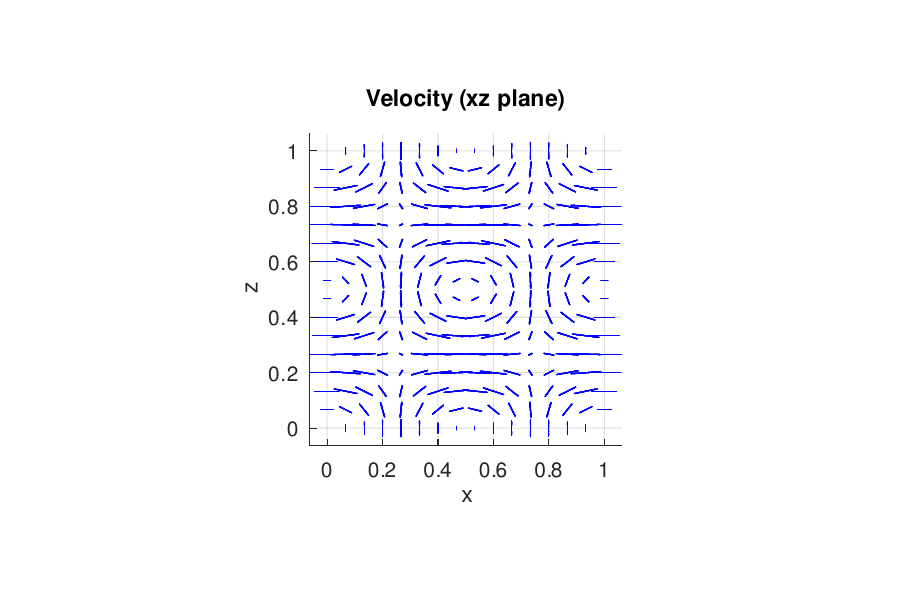} 
\hspace*{-40mm}
\includegraphics[scale=0.8]{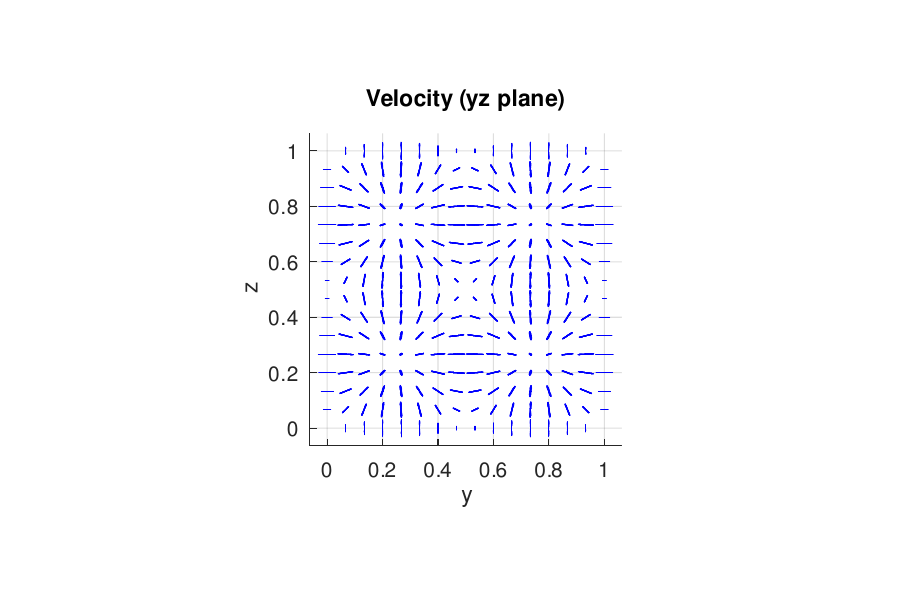} 
\\
\noindent
\hspace*{-22mm}
\includegraphics[scale=0.8]{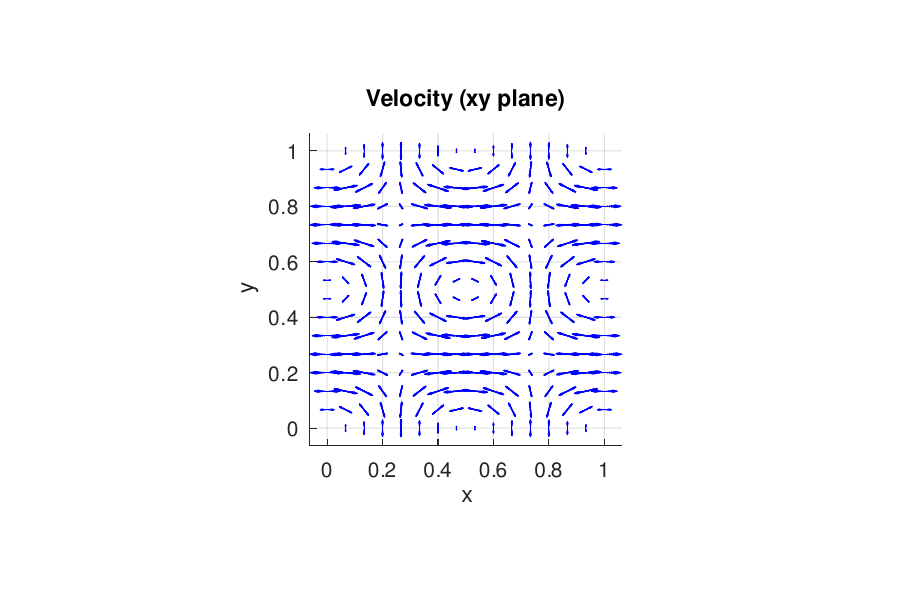} 
\hspace*{-35mm}
\includegraphics[scale=0.8]{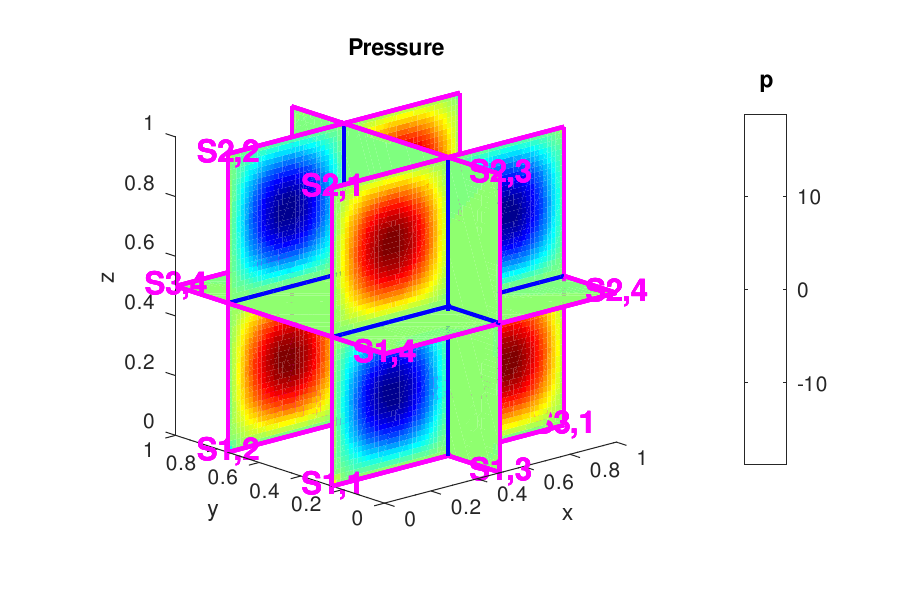} 
\caption{Results for Taylor-Green voxrtex problem with $n=40$  (so $h=0.025$).
Here 
$\|\bu_n-\bu\|_{L^2(\Omega)}/\|\bu\|_{L^2(\Omega)}=4.2\cdot10^{-3}$,
$\|\bu_n-\bu\|_{L^\infty(\Omega)}/\|\bu\|_{L^\infty(\Omega)}=7.2\cdot10^{-3}$,
$\|p_n-p\|_{L^2(\Omega)}/\|p\|_{L^2(\Omega)}=2.0\cdot 10^{-2}$,
$\|p_n-p\|_{L^\infty(\Omega)}/\|p\|_{L^\infty(\Omega)}=5.0\cdot 10^{-1}$,
$\|\nabla\cdot\bu_n\|_{L^2(\Omega)}=6.7\cdot10^{-2}$,
$\|\nabla\cdot\bu_n\|_{L^\infty(\Omega)}=4.5\cdot 10^{-1}$.
}
\label{f:tg-40}
\end{figure}

\section{Conclusions}\label{sec:concl}
We prove that the velocity is the superposition of two divergence-free velocity vector fields.
One of them is generated by the divergence-free  rotational component of the external force term,
and the other is generated by the gradient of a  ``solid pressure" harmonic function, chosen to match the normal component of the given boundary velocity.
It turns out that even the pressure is the superposition of the external force pressure and
the solid pressure.

The results leads to a novel numerical method for solving generalized Stokes equations.
The method is easily implemented numerically. It consists on solving systems 
of screened Poisson vector equations coupled only by some boundary conditions,
which moreover are completely uncoupled if the domain is a rectangular box.
We have presented a number of numerical results for some  two and three 
dimensional benchmark problems with known explicit solutions.
The numerical results are with remarkable precision.

The method applies to generalized Stokes equations with non-zero divergence and  with Neumann and mixed boundary conditions, which we will show in details in an upcoming paper.




\begin{thebibliography}{99}

%
%

\bibitem{angeli+al-2017} 
P.-E. Angeli, M.-A. Puscas, G. Fauchet, A. Cartalade,
``FVCA8 benchmark for the Stokes and Navier-Stokes equations with the TrioCFD code - 
benchmark session",
Finite Volumes for Complex Applications 8, Jun 2017, Lille, France




\bibitem{uzawa-1958} 
K. J. Arrow, L. Hurwicz, H. Uzawa,
``Studies in Linear and Nonlinear Programming",
Stanford University Press, 1958




\bibitem{boyer+al-2017} 
F. Boyer, P. Omnes,
``Benchmark proposal for the FVCA8 conference : 
Finite volume methods for the Stokes and Navier-Stokes equations",
Springer Proceedings in Mathematics \& Statistics, 199,
Springer, pp.59-71, 2017, 
FVCA 2017: Finite Volumes for Complex Applications VIII - 
Methods and Theoretical Aspects, 978-3-319-57396-0 


\bibitem{brezis-2011} 
H. Brezis,
``Functional Analysis, Sobolev Spaces and Partial Differential Equations",
Universitext, Springer, 2011

\bibitem{bristeau+al-1987}
M. O. Bristeau, R. Glowinski, J. P\'eriaux,
``Numerical methods for the Navier-STokes equations. 
Applications to the simulation of compressible and incompressible viscous flows",
Computer Physics Reports 6 (1987) 73-187
North-Holland, Amsterdam


\bibitem{chorin-1968} 
A. J. Chorin,
``Numerical Solution of the Navier-Stokes Equations", 
Math. Comp., 22 (104): 745–762, 1968




\bibitem{christensen-2016} 
O. Christensen,
``An Introduction to Frames and Riesz Bases",
Second Edition,
Birkh\"auser, 2016





\bibitem{russell-1982-ns+char} 
J. Douglas and T. F. Russell, 
``Numerical method for convection dominated diffusion problems based on combining the
method of characteristics with finite element or finite difference procedures", 
SIAM J. Numer. Anal., 19 (1982)







\bibitem{gilbarg+trudinger-2001} 
D. Gilbarg, N. S. Trudinger,
``Elliptic Partial Differential Equations of Second Order", 
Volume 224,
Springer Science \& Business Media, Jan 12, 2001


\bibitem{girault+raviart-1986} 
V. Girault, P.-A. Raviart,
``Finite Element Methods for Navier-Stokes Equations. Theory and Algorithms",
Springer-Verlag, Berlin, Heidelberg,
1986


\bibitem{glowinski+pironneau-1978}
R. Glowinski and O. Pironneau,
``On Numerical Methods for the Stokes Problem",
Publications des séminaires de math\'ematiques et informatique de Rennes, 1978, fascicule S4
Journ\'ees \'el\'ements finis, p. 1-29
%
\bibitem{glowinski+pironneau-1979}
R. Glowinski and O. Pironneau,
``On a Mixed Finite Element Approximation of the Stokes Problem (I). Convergence of the Approximate Solutions",
Numer. Math. 33, 397-424 (1979)


\bibitem{glowinski-1984} 
R. Glowinski, 
``Numerical Methods for Nonlinear Variational Problems", 
Springer-Verlag, New York, 1984
 

\bibitem{grisvard-1985} 
P. Grisvard,
``Elliptic Problems in Nonsmooth Domains",
Pitman Publishing Inc. 1020 Plain Street, Marshfield, Massachusetts 02050,
1985



\bibitem{guermond+al-2006} 
J. L. Guermond,P. Minev, J. Shen,
``An overview of projection methods for incompressible flows",
Comput. Methods Appl. Mech. Engrg. 195 (2006) 6011–6045





\bibitem{hecht-2012} 
F. Hecht,
``New development in FreeFem++",
Journal of numerical mathematics,
20.3-4, 251-266, 2012





\bibitem{lions+magenes-1961} 
J.-L. Lions, E. Magenes,
``Problem\`es aux limites non homog\`enes",
Annales de l'Institut Fourier, 11, 137-178, 1961



\bibitem{mclean-2000}
W. McLean, 
``Strongly Elliptic Systems and Boundary Integral Equations",
Cambridge University Press,
2000

\bibitem{pironneau-1982-ns+char} 
O. Pironneau, 
``On the transport diffusion algorithm and its applications to the Navier-Stokes equations", Numer. Math., 38, 309-332 (1982)


\bibitem{temam-1969} 
R. Temam, 
``Sur l'approximation de la solution des \'equations de Navier-Stokes par la
m\'ethode des fractionnaires II", 
Arch. Rational Mech. Anal. 33 (1969), 377-385

\bibitem{temam-2001} 
R. Temam,
``Navier–Stokes equations : theory and numerical analysis",
AMS Chelsea Publishing,
American Mathematical Society, Providence, Rhode Island,
2001






\bibitem{wahl-1992} 
W. von Wahl,
``Estimating $\nabla\bu$ by ${\rm div}\bu$ and ${\rm curl}\bu$",
Mathematical Methods in the Applied Sciences, Vol. 15, 123-143 (1992)

\end{thebibliography}
\end{document}